%% file: main.tex
\pdfoutput=1
\documentclass[nohyperref]{article}

\newcommand*{\arxiv}{}%

\ifdefined\arxiv
\usepackage{fullpage}
\fi

\usepackage{microtype}
\usepackage{graphicx}
\usepackage{subfigure}
\usepackage{booktabs} %
\usepackage[T1]{fontenc}
\usepackage{placeins}
\usepackage{algorithm}
\usepackage{algorithmic}

\input{defs}

\usepackage{hyperref}

\ifdefined\arxiv
\else
\usepackage{icml2022}

\fi

\usepackage{amsmath}
\usepackage{amssymb}
\usepackage{mathtools}
\usepackage{amsthm}

\usepackage[capitalize,noabbrev]{cleveref}

\theoremstyle{plain}
\newtheorem{theorem}{Theorem}[section]

\newtheorem{lemma}[theorem]{Lemma}

\theoremstyle{definition}

\theoremstyle{remark}

\makeatletter
\newtheorem*{rep@theorem}{\rep@title}
\newcommand{\newreptheorem}[2]{%
\newenvironment{rep#1}[1]{%
 \def\rep@title{\bf #2 \ref{##1}}%
 \begin{rep@theorem}}%
 {\end{rep@theorem}}}
\makeatother

\newreptheorem{theorem}{Theorem}
\newreptheorem{lemma}{Lemma}

\usepackage[textsize=tiny]{todonotes}

\ifdefined\arxiv
\else
\icmltitlerunning{Iterative Hard Thresholding with Adaptive Regularization: Sparser
Solutions Without Sacrificing Runtime}
\fi

\ifdefined\arxiv
\title{Iterative Hard Thresholding with Adaptive Regularization: Sparser Solutions Without Sacrificing Runtime}
\author{
Kyriakos Axiotis\thanks{MIT, \tt{kaxiotis@mit.edu}}
\and
Maxim Sviridenko\thanks{Yahoo! Research, \tt{sviri@yahooinc.com}}
\date{}
}
\fi

\usepackage{pythonhighlight}

\begin{document}

\ifdefined\arxiv
\maketitle
\else
\twocolumn[
\icmltitle{Iterative Hard Thresholding with Adaptive Regularization: Sparser
Solutions Without Sacrificing Runtime}

\begin{icmlauthorlist}
\icmlauthor{Kyriakos Axiotis}{MIT}
\icmlauthor{Maxim Sviridenko}{Yahoo}
\end{icmlauthorlist}

\icmlaffiliation{MIT}{MIT}
\icmlaffiliation{Yahoo}{Yahoo! Research}

\icmlcorrespondingauthor{Kyriakos Axiotis}{kaxiotis@mit.edu}
\icmlcorrespondingauthor{Maxim Sviridenko}{sviri@yahooinc.com}

\icmlkeywords{Machine Learning, ICML}

\vskip 0.3in
]
\fi

\ifdefined\arxiv
\else
\printAffiliationsAndNotice{}  %
\fi

\begin{abstract}

We propose a simple modification to the iterative hard thresholding (IHT) algorithm, which recovers
asymptotically sparser solutions as a function of the condition number.
When aiming to minimize a convex function $f(\xx)$ with condition number $\kappa$ subject to $\xx$ being an $s$-sparse vector,
the standard IHT guarantee is a solution with relaxed sparsity $O(s\kappa^2)$, while our proposed algorithm, \emph{regularized IHT},
returns a solution with sparsity $O(s\kappa)$. Our algorithm 
significantly improves over ARHT~\cite{axiotis2021sparse} which also finds a solution of sparsity $O(s\kappa)$,
as it does not require re-optimization in each iteration (and so is much faster), is deterministic, and does not require knowledge of the optimal solution value
$f(\xx^*)$ or the optimal sparsity level $s$.

Our main technical tool is an \emph{adaptive regularization} framework, in which the algorithm progressively learns the weights
of an $\ell_2$ regularization term that will allow convergence to sparser solutions. We also apply this framework to low rank
optimization, where we achieve a similar improvement of the best known condition number dependence from $\kappa^2$ to $\kappa$.
\end{abstract}

\input{intro}
\input{background}
\input{adaptive_regularization}
\input{sparse_optimization}
\input{lowrank_optimization}
\FloatBarrier
\input{experiments}
\FloatBarrier

\bibliography{main}
\bibliographystyle{alpha}
\ifdefined\arxiv
\else
\bibliographystyle{icml2022}
\fi

\newpage
\appendix
\input{appendix}
\input{rank_final}
\input{appendix_lower_bound}

\end{document}

%% file: defs.tex
\newcommand{\cA}{\mathcal{A}}
\newcommand{\rank}{\mathrm{rank}}
\newcommand{\im}{\mathrm{im}}

\newcommand{\onev}{\mathbf{1}}
\newcommand{\zerov}{\mathbf{0}}

\newcommand{\tO}[1]{\widetilde{O}\left(#1\right)}

\newcommand{\ox}{\bar{x}}

\newcommand{\olam}{\bar{\lambda}}

\newcommand{\eps}{\varepsilon}

\newcommand{\ow}{\bar{w}}

\newcommand{\vPi}{\boldsymbol{\mathit{\Pi}}}

\newcommand{\vSigma}{\boldsymbol{\mathit{\Sigma}}}

\newcommand{\vlambda}{\boldsymbol{\lambda}}
\newcommand{\vLambda}{\boldsymbol{\Lambda}}

\renewcommand{\aa}{\boldsymbol{\mathit{a}}}
\newcommand{\bb}{\boldsymbol{\mathit{b}}}

\newcommand{\uu}{\boldsymbol{\mathit{u}}}

\newcommand{\vv}{\boldsymbol{\mathit{v}}}
\newcommand{\ww}{\boldsymbol{\mathit{w}}}
\newcommand{\oww}{\boldsymbol{\mathit{\bar{w}}}}
\newcommand{\xx}{\boldsymbol{\mathit{x}}}
\newcommand{\oxx}{\boldsymbol{\bar{\mathit{x}}}}
\newcommand{\ovv}{\boldsymbol{\bar{\mathit{v}}}}

\newcommand{\yy}{\boldsymbol{\mathit{y}}}

\newcommand{\zz}{\boldsymbol{\mathit{z}}}

\renewcommand{\AA}{\boldsymbol{\mathit{A}}}
\newcommand{\oAA}{\boldsymbol{\overline{\mathit{A}}}}
\newcommand{\BB}{\boldsymbol{\mathit{B}}}
\newcommand{\CC}{\boldsymbol{\mathit{C}}}

\newcommand{\II}{\boldsymbol{\mathit{I}}}

\newcommand{\MM}{\boldsymbol{\mathit{M}}}

\newcommand{\OO}{\boldsymbol{\mathit{O}}}
\newcommand{\PP}{\boldsymbol{\mathit{P}}}
\newcommand{\QQ}{\boldsymbol{\mathit{Q}}}

\newcommand{\UU}{\boldsymbol{\mathit{U}}}
\newcommand{\VV}{\boldsymbol{\mathit{V}}}
\newcommand{\WW}{\boldsymbol{\mathit{W}}}
\newcommand{\XX}{\boldsymbol{\mathit{X}}}
\newcommand{\YY}{\boldsymbol{\mathit{Y}}}

\renewcommand{\epsilon}{\varepsilon}

\usepackage[dvipsnames]{xcolor}

%% file: intro.tex
\section{Introduction}

\emph{Sparse optimization} is the task of optimizing a function $f$ over $s$-sparse vectors, i.e.
those with at most $s$ non-zero entries.
Examples of such 
optimization problems arise in machine learning, with the goal to make models smaller
for efficiency, generalization, or interpretability reasons,
and compressed sensing, where the goal is to recover an $s$-sparse signal from a small number of 
measurements. A closely related problem is \emph{low rank optimization}, where the sparsity
constraint is instead placed on the spectrum of the solution (which is a matrix). This problem is central
in matrix factorization, recommender systems, robust principal components analysis, among other tasks.
More generally, \emph{structured sparsity} constraints have the goal of capturing the special structure
of a particular task by restricting the set of solutions to those that are ``simple'' in an appropriate sense.
Examples include group sparsity, tree- and graph-structured sparsity. For more on 
generalized sparsity measures see e.g.~\cite{schmidt18}.

Among the huge number of algorithms that have been developed for the sparse optimization problems,
three stand out as the most popular ones:
\begin{itemize}
\item {The {\bf LASSO}~\cite{LASSO}, which works by relaxing the $\ell_0$ (sparsity)
constraint to an $\ell_1$ constraint, thus convexifying the problem.}
\item{{\bf Orthogonal matching pursuit} (OMP)~\cite{pati1993orthogonal}, which works by building
the solution greedily in an incremental fashion.}
\item{{\bf Iterative hard thresholding} (IHT)~\cite{IHT}, which performs projected gradient
descent on the set of sparse solutions.}
\end{itemize}
Among these, IHT is generally the most efficient, since it has essentially no overhead
over plain gradient descent, making it the tool of choice for large-scale applications. 

\subsection{Iterative Hard Thresholding (IHT)}

Consider the \emph{sparse convex optimization} problem
\begin{align}
\underset{\left\|\xx\right\|_0 \leq s}{\min}\, f(\xx)\,,
\label{eq:objective_intro}
\end{align}
where $f$ is convex and $\left\|\xx\right\|_0$ is the number of non-zero entries in the vector $\xx$,
i.e. the sparsity of $\xx$.
IHT works by repeatedly performing the following iteration
\begin{align}
\xx^{t+1} = H_{s'}\left(\xx^{t} - \eta\cdot \nabla f(\xx^{t})\right)\,,
\label{eq:IHT_step}
\end{align}
where $H_{s'}$ is the \emph{hard thresholding operator} that zeroes out all but the top $s'$ entries, for some
(potentially relaxed) sparsity level $s'$, and $\eta > 0$ is the step size.

As (\ref{eq:objective_intro}) is known to be NP-hard~\cite{Natarajan95} and even hard to approximate~\cite{FKT15},
an extra assumption needs to be made for the performance of the algorithm to be theoretically evaluated
in a meaningful way.
The most common assumption is that the \emph{(restricted) condition number}
of $f$ is bounded by $\kappa$
(or the \emph{restricted isometry property} constant is bounded by $\delta$~\cite{Candes08}), but
other assumptions have been studied as well, 
such as incoherence~\cite{donoho2003optimally}
and weak supermodularity~\cite{liberty2017greedy}.
The performance is then measured in terms of the sparsity $s'$ of the returned solution, 
as well as its error (value of $f$).

As it is known~\cite{JTK14}, IHT is guaranteed to return an $s'=O(s\kappa^2)$-sparse solution $\xx$ 
with $f(\xx)\leq f(\xx^*)+\eps$.
In fact, as we show in Section~\ref{sec:lower_bounds}, 
the $\kappa^2$ factor cannot be improved in the analysis.
Recently, \cite{axiotis2021sparse} presented an algorithm called ARHT,
which improves the sparsity to $s'=O(s\kappa)$. However, their algorithm is 
much less efficient than IHT, for many reasons.
So the question emerges:

\ifdefined\arxiv
\vskip 0.75cm
\fbox{
\begin{minipage}{40em}
\fi
\emph{Is there a sparse convex optimization algorithm that returns $O(s\kappa)$-sparse solutions, 
but whose runtime efficiency is comparable to IHT?}
\ifdefined\arxiv
\end{minipage}
}
\vskip 0.75cm
\fi

The main contribution of our work is to show that this goal can be achieved, and done so
by a surprisingly simple tweak to IHT.

\subsection{Reconciling Sparsity and Efficiency: Regularized IHT}

Our main result is the following theorem, which states that 
running IHT on an \emph{adaptively regularized} objective function returns $O(s\kappa)$-sparse
solutions that are $\eps$-optimal in function value, 
while having no significant runtime overhead over plain gradient descent.

\begin{theorem}[Regularized IHT]
Let $f\in\mathbb{R}^{n}\rightarrow\mathbb{R}$ be a convex function that is $\beta$-smooth
and $\alpha$-strongly convex\footnote{The theorem also holds if 
the smoothness and strong convexity constants are replaced by $(s'+s)$-\emph{restricted} smoothness and strong convexity constants.}, with
condition number $\kappa=\beta/\alpha$, and $\xx^*$
be an (unknown) $s$-sparse solution. Then,
running 
Algorithm~\ref{alg:regularized_iht} with $\eta = (2\beta)^{-1}$ and $c=s'/(4T)$
for 
\[ T = O\left(\kappa \log \frac{f(\xx^0) + (\beta/2)\left\|\xx^0\right\|_2^2 - f(\xx^*)}{\eps}\right) \]
iterations starting from an arbitrary  $s'=O(s\kappa)$-sparse solution $\xx^0$, 
the algorithm returns an $s'$-sparse solution $\xx^T$ such that
$f(\xx^T) \leq f(\xx^*) + \eps$.
Furthermore, each iteration requires
$O(1)$ evaluations of $f$, $\nabla f$, and $O(n)$ additional time.
\label{thm:regularized_iht}
\end{theorem}

To achieve this result, we significantly refine and generalize the \emph{adaptive regularization} technique of~\cite{axiotis2021sparse}. 
This refined version fixes many of the shortcomings of the original, 
by (i) not requiring \emph{re-optimization} in every iteration (a relic of OMP-style algorithms),
(ii) taking $\tO{\kappa}$ instead of $\tO{s\kappa}$ iterations,
(iii) being \emph{deterministic}, 
(iv) not requiring knowledge of the optimal function value $f(x^*)$ thus avoiding the overhead of an outer binary search,
and (v) being more easily generalizable to other settings, like low rank minimization.

In short, our main idea is to run IHT on a regularized function 
\[ g(\xx) = f(\xx) + (\beta/2) \left\|\xx\right\|_{\ww,2}^2\,,\] 
where $\left\|\xx\right\|_{\ww,2}^2=\sum_{i=1}^nw_ix_i^2$ and $\ww$ are non-negative weights.
These weights change dynamically during the algorithm, in a way that depends on the value of $\xx$.
The effect is that now the IHT step will instead be given by
\begin{align*}
\xx^{t+1} 
& = H_{s'}\left(\left(\onev - 0.5 \ww^{t}\right)\xx^{t} - \eta \cdot \nabla f(\xx^{t})\right)\,,
\end{align*}
which is almost the same as (\ref{eq:IHT_step}), except that it has an extra term 
that biases the solution towards $\zerov$.
Additionally, in each step the weights $\ww^t$ are updated based on the current solution as
\begin{align*}
w^{t+1}_i = \left(w^{t}_i \cdot \left(\onev - 
c\cdot \frac{w^{t}_i (x^t_i)^2}{\left\|\xx^t\right\|_{w^{t},2}^2}\right)\right)_{\geq 1/2}
\end{align*}
for some parameter $c > 0$, where $(\cdot)_{\geq 1/2}$ denotes zeroing out all the entries that are $<1/2$ and keeping the others intact.

In Section~\ref{sec:adaptive_regularization}, we will go over the central ideas 
of our refined adaptive regularization technique, and also explain how it can be extended to
deal with more general sparsity measures.

\subsection{Beyond Sparsity: Low Rank Optimization}

As discussed, our new techniques transfer to the problem of minimizing a convex function 
under a rank constraint. In particular, we prove the following theorem:

\begin{theorem}[Adaptive Regularization for Low Rank Optimization]
Let $f\in\mathbb{R}^{m\times n}\rightarrow\mathbb{R}$ be a
convex function with condition number $\kappa$ and
consider the low rank minimization
problem
\begin{align}
\underset{\rank(\AA) \leq r}{\min}\, f(\AA)\,.
\label{eq:rank_objective_intro}
\end{align}
For any error parameter $\eps>0$,
there exists a polynomial time algorithm that returns a matrix $\AA$ with
$\rank(\AA) \leq O\left(r\left(\kappa+\log \frac{f(\OO)-f(\AA^*)}{\eps}\right)\right)$
and 
$f(\AA) \leq f(\AA^*) + \eps$, where 
$\OO$ is the all-zero matrix and 
$\AA^*$ is any rank-$r$ matrix.
\label{thm:rank_ARHT_plus}
\end{theorem}

This result can be compared to the Greedy algorithm of~\cite{axiotis2021local},
which works by incrementally adding a rank-$1$ component to the solution and achieves
rank $O(r\kappa \log \frac{f(\OO) - f(\AA^*)}{\epsilon})$, as well as their
Local Search algorithm, which works by simultaneously adding a rank-$1$ component and 
removing another, and achieves rank $O(r\kappa^2)$.
In contrast, our Theorem~\ref{thm:rank_ARHT_plus} returns a solution with
rank $O\left(r\left(\kappa + \log \frac{f(\OO)-f(\AA^*)}{\epsilon}\right)\right)$.

\subsection{Related Work}

The sparse optimization and compressed sensing literature has a wealth of different
algorithms and analyses. Examples include the seminal paper of \cite{Candes08} on recovery
with LASSO and followup works~\cite{Foucart10}, 
the CoSaMP algorithm~\cite{CoSaMP}, 
orthogonal matching pursuit and variants~\cite{Natarajan95,SSZ10,JTD11,axiotis2021sparse}
iterative hard thresholding~\cite{IHT,JTK14}, 
hard thresholding pursuit~\cite{foucart2011hard,YLZ16,SL17,SL17_2},
partial hard thresholding~\cite{JTD17},
and message passing algorithms~\cite{donoho2009message}.
For a survey on compressed sensing, see \cite{BCKV15,foucart2017mathematical}.

A family of algorithms that is closely related to IHT
are Frank-Wolfe (FW) methods~\cite{frank1956algorithm},
which have been used for dealing with generalized sparsity constraints~\cite{jaggi2013revisiting}.
The basic
version can be viewed as a variant of OMP without re-optimization in each iteration. 
Block-FW methods are more resemblant of IHT without the projection step, see e.g. 
\cite{allen2017linear} for an application to the low rank minimization problem.

\cite{liu2020between} presented an interesting connection between hard and soft thresholding algorithms by studying a
concavity property of the thresholding operator, and proposed new thresholding operators.

Recently it has been shown~\cite{peste2021ac} that IHT can be guaranteed to work for sparse optimization
of \emph{non-convex} functions, under appropriate assumptions. In particular, \cite{peste2021ac}
studies a stochastic version of IHT for sparse deep learning
problems, from both a theoretical and practical standpoint.

%% file: background.tex
\section{Background}
\paragraph{Notation.} We denote $[n] = \{1,2,\dots,n\}$. We will use {\bf bold} to refer to vectors or matrices. We denote by $\zerov$ the all-zero vector, $\onev$ the all-one vector,
$\OO$ the all-zero matrix, and by $\II$ the identity matrix (with dimensions understood from the context). Additionally, we will denote by $\onev_i$ the $i$-th basis vector, i.e.
the vector that is $0$ everywhere except at position $i$.

In order to ease notation and where not ambiguous for two vectors $\xx, \yy\in \mathbb{R}^n$, we denote by $\xx \yy\in \mathbb{R}^n$ a vector with elements $(\xx\yy)_i=x_iy_i$, i.e.
the element-wise multiplication of two vectors $\xx$ and $\yy$. In contrast,
we denote their inner product by $\langle \xx,\yy\rangle$ or $\xx^\top \yy$. Similarly, $\xx^2\in \mathbb{R}^n$ will be the element-wise square of vector $\xx$.

\paragraph{Restrictions and Thresholding.}
For any vector $\xx\in\mathbb{R}^n$ and set $S\subseteq[n]$, we denote by
$\xx_S$ the vector that results from $\xx$ after zeroing out all the entries except those in positions given by indices in $S$. 
For any $t\in\mathbb{R}$ and $\xx\in\mathbb{R}^n$, we denote by $\xx_{\geq t}$ the vector that results from setting all the entries of $\xx$ that are less than
$t$ to $0$.
For a function $f\in\mathbb{R}^n\rightarrow\mathbb{R}$, its gradient $\nabla f(\xx)$, and a set of indices $S\subseteq [n]$, we denote $\nabla_S f(\xx) = (\nabla f(\xx))_S$.
We define the \emph{thresholding operator} $H_s(\xx)$
for any vector $\xx$ as $\xx_S$, where $S$ are the $s$ entries of $\xx$ with largest absolute value (breaking ties arbitrarily). We override the thresholding operator $H_r(\AA)$ when the argument is
a matrix $\AA$, defining $H_r(\AA) = \UU \mathrm{diag}\left(H_r(\vlambda)\right)\VV^\top$, where $\UU \mathrm{diag}(\vlambda) \VV^\top$ is the singular value decomposition of $\AA$, i.e.
$H_r(\AA)$ only keeps the top $r$ singular components of $\AA$.

\paragraph{Norms and Inner Products.}
For any $p\in(0,\infty)$ and weight vector $\ww\geq \zerov$, we define the weighted $\ell_p$ norm of a vector $\xx\in\mathbb{R}^n$ as:
\begin{align*}
\left\|\xx\right\|_{p,\ww} = \left(\sum\limits_{i} w_i x_i^p \right)^{1/p}\,.
\end{align*}
For $p=0$, we denote $\left\|\xx\right\|_0 = \left|\{i\ |\ x_i\neq 0\}\right|$ to be the \emph{sparsity} of $\xx$.
For $p=\infty$, we denote $\left\|\xx\right\|_\infty = \max_i |x_i|$ to be the maximum absolute value of $\xx$.

For a matrix $\AA\in\mathbb{R}^{m\times n}$, we let $\left\|\AA\right\|_2$ be its spectral norm, $\left\|\AA\right\|_F$ be its Frobenius norm,
and $\left\|\AA\right\|_*$ be its nuclear norm (i.e. sum of singular values). For any $\BB\in\mathbb{R}^{m\times n}$, we denote the Frobenius inner product
as $\langle \AA,\BB\rangle = \mathrm{Tr}\left[\AA^\top \BB\right]$.

\paragraph{Smoothness, strong convexity, condition number.}
A differentiable function $f:\mathbb{R}^n\rightarrow\mathbb{R}$ is called \emph{convex} if for any
$\xx,\yy\in\mathbb{R}^n$ we have $f(\yy) \geq f(\xx) + \langle \nabla f(\xx), \yy-\xx\rangle$.
Furthermore, 
$f$ is called \emph{$\beta$-smooth} for some real number $\beta > 0$ if for any $\xx,\yy\in\mathbb{R}^n$
we have
$f(\yy) \leq f(\xx) + \langle \nabla f(\xx), \yy-\xx\rangle + (\beta/2) \left\|\yy-\xx\right\|_2^2$
and \emph{$\alpha$-strongly convex} for some real number $\alpha > 0$ if for any 
$\xx,\yy\in\mathbb{R}^n$ we have
$f(\yy) \geq f(\xx) + \langle \nabla f(\xx), \yy-\xx\rangle + (\alpha/2) \left\|\yy-\xx\right\|_2^2$.
We call $\kappa:=\beta/\alpha$ the \emph{condition number} of $f$. If $f$ is only $\beta$-smooth
along $s$-sparse directions (i.e. only for $\xx,\yy\in\mathbb{R}^n$ such that 
$\left\|\yy-\xx\right\|_0 \leq s$), then we call $f$ $\beta$-smooth \emph{at sparsity level $s$}
and denote the smallest such $\beta$ by $\beta_s$ and call it the 
\emph{restricted smoothness constant} (at sparsity level $s$).
We analogously define the \emph{restricted strong convexity constant} $\alpha_s$, as well as the
$s$-\emph{restricted condition number} $\kappa_s := \beta_s / \alpha_s$.

\paragraph{Projections.}

Given a subspace $\mathcal{V}$, we will denote the orthogonal projection onto $\mathcal{V}$ as 
$\vPi_{\mathcal{V}}$. In particular, for any matrix $\AA\in\mathbb{R}^{m\times n}$ we denote by $\im(\AA) = \{\AA\xx\ |\ \xx\in\mathbb{R}^n\}$ the \emph{image} of $\AA$
and by $\ker(\AA) = \{\xx\ |\ \AA^\top \xx = \zerov\}$ the \emph{kernel} of $\AA$. Therefore, $\vPi_{\im(\AA)} = \AA\left(\AA^\top \AA\right)^+\AA^\top$ is the orthogonal projection
onto the image of $\AA$ and $\vPi_{\ker(\AA^\top)} = \II - \vPi_{\im(\AA)}$ the orthogonal projection onto the kernel of $\AA^\top$, where $(\cdot)^+$ denotes the matrix pseudoinverse.

%% file: adaptive_regularization.tex
\section{The Adaptive Regularization Method}
\label{sec:adaptive_regularization}

Consider the sparse optimization problem
\begin{align}
\underset{\left\|\xx\right\|_0 \leq s}{\min}\, f(\xx)
\label{eq:objective}
\end{align}
on a convex function $f$ with condition number at most $\kappa$,
and an optimal solution $\xx^*$ that is supported on the set of indices $S^*\subseteq [n]$.

The main hurdle towards solving this problem is that it is NP hard.
Therefore, it is common to relax it by a factor depending on $\kappa$.
In fact, IHT requires relaxing the sparsity constraint by a factor of $O(\kappa^2)$ (i.e. $\left\|\xx\right\|_0 \leq O(s \kappa^2)$), in order to return a near-optimal solution.
Also, the $\kappa^2$ factor is tight for IHT (see Appendix~\ref{sec:lower_bounds}).

\paragraph{Remark} We state all our results in terms of the condition number $\kappa$, even though the statements can be strengthened
to depend on the \emph{restricted} condition number $\kappa_{s'+s}$, specifically the condition number restricted on $(s'+s)$-sparse directions. We state our results in this weaker form for clarity of presentation.

\subsection{Regularized IHT}
Perhaps surprisingly, there is a way to \emph{regularize} the objective by a weighted $\ell_2$ norm 
so that running IHT on the new objective will only require
relaxing the sparsity by $O(\kappa)$:
\begin{align}
\underset{\left\|\xx\right\|_0 \leq s}{\min}\, f(\xx) + (\beta/2) \left\|\xx\right\|_{\ww,2}^2\,.
\label{eq:regularized_objective}
\end{align}
One way to do this is by setting the weights $\ww$ to be $1$ everywhere except in the indices from $S^*$, where it is set to $0$.
An inquisitive reader will protest that this is not a very useful statement, since it requires knowledge of $S^*$, which was our goal to begin with.
In fact, we could just as easily have used the regularizer $(\beta/2) \left\|\xx - \xx^*\right\|_{2}^2$, thus penalizing everything that is far from the optimum!

\subsection{Learning Weights}
Our main contribution is to show that the optimal weights $\ww$ can in fact be \emph{learned} in the 
duration of the algorithm\footnote{The idea of adaptively learning regularization weights
looks on the surface similar to adaptive gradient algorithms such as
AdaGrad~\cite{duchi2011adaptive}. An important difference is that these algorithms
regularize the function around
the current solution, while we regularize it around the origin. Still, this is a potentially
intriguing connection that deserves to be investigated further.}.
More precisely, consider running IHT starting from the setting of $\ww=\onev$. The regularized objective (\ref{eq:regularized_objective}) is now
$O(1)$-conditioned, which is great news. On the other hand, (\ref{eq:regularized_objective}) is not what we set out to minimize.
In other words, even though this approach might work great for minimizing (\ref{eq:regularized_objective}), it might (and generally will) fail to achieve sufficient decrease
in (\ref{eq:objective})---one could view this as the algorithm getting trapped in a local minimum.

Our main technical tool is to characterize these local minima, by showing that they can only manifest themselves 
if the current solution $\xx$ satisfies the following condition:
\begin{align}
\left\|\xx_{S^*}\right\|_{\ww,2}^2 \geq \Omega(\kappa^{-1}) \left\|\xx\right\|_{\ww,2}^2\,.
\label{eq:nonprogress}
\end{align}
In words, this means that a significant fraction of the mass of the current solution lies in the support $S^*$ of the optimal solution.
Interestingly, this gives us enough information based on which to update the regularization weights $\ww$ in a way that the sum of weights in
$S^*$ drops fast enough compared to the total sum of weights. This implies that the vector $\ww$ moves in a direction that correlates with the direction of the optimal weight vector.

These are the core ideas needed to bring the sparsity overhead of IHT from $O(\kappa^2)$ down to $O(\kappa)$.

\subsection{Beyond Sparsity: Learning Subspaces}

One can summarize the approach of the previous section in the following more general way: If we know that the optimal solution $\xx^*$ lies in a particular low-dimensional subspace
(in our case this was the span of $\onev_i$ for all $i\in S^*$), then we can define a regularization term that penalizes all the solutions based on their distance to that subspace.
Of course, this subspace is unknown to us, but we can try to adaptively modify the regularization term every time the algorithm gets stuck, 
just as we did in the previous section.

More concretely, given a collection $\cA$ of unit vectors from $\mathbb{R}^n$ (commonly called \emph{atoms}), we define the following problem:
\begin{align}
\underset{\mathrm{rank}_{\cA}(\xx) \leq r}{\min}\, f(\xx)\,,
\label{eq:objective2}
\end{align}
where $\mathrm{rank}_{\cA}(\xx)$ is the smallest number of vectors from $\cA$ such that $\xx$ can be written as their linear combination. We can pick
$\cA = \{\onev_1,\onev_2,\dots,\onev_n\}$ to obtain the sparse optimization problem,
$\cA = \{\mathrm{vec}(\uu\vv^\top) \ |\ \left\|\uu\right\|_2 = \left\|\vv\right\|_2 = 1\}$ for the low rank minimization problem, and other choices of $\cA$
can capture more sophisticated problem constraints such as graph structure.
Defining an IHT variant for these more general settings is usually straightforward, although the analysis for even obtaining a rank overhead of $O(\kappa^2)$ 
does not trivially follow and depends on the structure of $\cA$.

So, how would a regularizer look in this more general setting? Given our above discussion, it is fairly simple to deduce it. Consider a decomposition of $\xx$ as
the sum of rank-$1$ components from $\cA = \{\aa_1,\aa_2,\dots,\aa_{|\cA|}\}$:
\[ \xx=\sum\limits_{i\in S} \aa_i \,, \]
where $\rank_{\cA}(\aa_i) = 1$, and let $L^* = \mathrm{span}(\{\aa_i\ |\ i\in S^*\})$ be a low-dimensional subspace that contains the optimal solution
and $L_{\perp}^*$ is its complement. We can then define the regularizer
\begin{align*}
\Phi^*(\xx) = (\beta/2) \sum\limits_{i\in S} \left\|\vPi_{L_{\perp}^*} \aa_i\right\|_2^2\,,
\end{align*}
where $\vPi_{L_{\perp}^*}$ is the orthogonal projection onto the subspace perpendicular to $L^*$---in other words $\left\|\vPi_{L_{\perp}^*} \aa_i\right\|_2$ is the $\ell_2$ distance
from $\aa_i$ to $L^*$.  An equivalent but slightly more concise way is to write:
\begin{align*}
\Phi^*(\xx) 
& = (\beta/2) \left\langle\vPi_{L_{\perp}^*}, \sum\limits_{i\in S} \aa_i \aa_i^\top\right\rangle\,.
\end{align*}
Then, we can replace the unknown projection matrix $\vPi_{L_{\perp}^*}$ by a weight matrix $\WW$ initialized at $\II$, and proceed by adaptively modifying $\WW$ as we did in the previous section.

It should be noted that the full analysis of this framework is not automatic for general $\cA$, 
and there are several technical challenges that arise depending on the choice of $\cA$. 
In particular, it does not directly apply to the low rank minimization case, and we end up using a different choice of regularizer.
However, the discussion in this section
should serve as a basic framework for improving the IHT analysis in more general settings, 
as in particular it did to motivate 
the low rank optimization analysis that we will present in Section~\ref{sec:lowrank_optimization}.

%% file: sparse_optimization.tex
\section{Sparse Optimization Using Regularized IHT}

The main result of this section is an efficient algorithm for sparse optimization of convex functions that,
even though is a slight modification of IHT, improves the sparsity by an $O(\kappa)$ factor, where $\kappa$ is the condition number.
The regularized IHT algorithm is presented in Algorithm~\ref{alg:regularized_iht} and its analysis is in Theorem~\ref{thm:regularized_iht}, whose proof can be found in Appendix~\ref{sec:proof_regularized_iht}.

\begin{algorithm}%
   \caption{Regularized IHT}
   \label{alg:regularized_iht}
\begin{algorithmic}
\STATE $\xx^0$: initial $s'$-sparse solution
\STATE $\ww^0 = \onev$: initial regularization weights
\STATE $\eta$: step size, $T$: \#iterations
\STATE $c$: weight step size
\FOR{$t=0\dots T-1$}
	\STATE $\xx^{t+1} = H_{s'}\left((\onev-0.5\ww^{t}) \xx^{t} - \eta \cdot \nabla f(\xx^{t})\right)$
	\STATE $\ww^{t+1} = \left(\ww^{t} - c \cdot (\ww^{t}\xx^{t})^2 / \left\|\xx^{t}\right\|_{\ww^{t},2}^2\right)_{\geq 1/2}$
	\IF {$f(\xx^{t+1}) + (4\eta)^{-1}\left\|\xx^{t+1}\right\|_{\ww^{t+1},2}^2 > f(\xx^{t}) + (4\eta)^{-1}\left\|\xx^{t}\right\|_{\ww^{t+1},2}^2$}
		\STATE $\xx^{t+1} = \xx^t$ \COMMENT{\textcolor{gray}{In practice there is no need to perform this step.}}
	\ENDIF
\ENDFOR
\end{algorithmic}
\end{algorithm}

\begin{reptheorem}{thm:regularized_iht}[Regularized IHT]
Let $f\in\mathbb{R}^{n}\rightarrow\mathbb{R}$ be a convex function that is $\beta$-smooth
and $\alpha$-strongly convex, with
condition number $\kappa=\beta/\alpha$, and $\xx^*$
be an (unknown) $s$-sparse solution. Then,
running 
Algorithm~\ref{alg:regularized_iht} with $\eta = (2\beta)^{-1}$ and $c=s'/(4T)$
for 
\[ T = O\left(\kappa \log \frac{f(\xx^0) + (\beta/2)\left\|\xx^0\right\|_2^2 - f(\xx^*)}{\eps}\right) \]
iterations starting from an arbitrary $s'=O(s\kappa)$-sparse solution $\xx^0$, 
the algorithm returns an $s'$-sparse solution $\xx^T$ such that
$f(\xx^T) \leq f(\xx^*) + \eps$.
Furthermore, each iteration requires
$O(1)$ evaluations of $f$, $\nabla f$, and $O(n)$ additional time.
\end{reptheorem}

The main ingredient for proving Theorem~\ref{thm:regularized_iht} is Lemma~\ref{lem:one_step},
which states that each step of the algorithm
either makes substantial (multiplicative) progress 
in an appropriately regularized function $f(\xx) + (\beta/2)\left\|\xx\right\|_{\ww,2}^2$, or a significant fraction of the mass of $\xx^2$
lies in $S^*$, which is the support of the target solution. This latter condition allows us to adapt the weights $\ww$
in order to obtain a new regularization function that penalizes the target solution less.
The proof of the lemma can be found in Appendix~\ref{sec:proof_lem_one_step}.
\begin{lemma}[Regularized IHT step progress]
Let $f\in\mathbb{R}^{n}\rightarrow\mathbb{R}$ be a convex 
function that is $\beta$-smooth and $\alpha$-strongly convex, 
$\kappa=\beta/\alpha$ be its condition number, and $\xx^*$ be any $s$-sparse
solution.

Given any $s'$-sparse solution $\xx\in\mathbb{R}^n$ where
\begin{align*}
s' \geq (128\kappa+2)s
\end{align*}
and a weight vector $\ww\in\left(\{0\}\cup[1/2,1]\right)^n$ 
such that $\left\|\ww\right\|_1 \geq n - s'/2$, we make the following update:
\[ \xx' = H_{s'}\left((\onev - 0.5 \ww) \xx - (2\beta)^{-1} \nabla f(\xx)\right) \,. \]
Then, at least one of the following two conditions holds:
\begin{itemize}
\item {
Updating $\xx$ makes regularized progress:
\[ g(\xx') \leq g(\xx) - (16\kappa)^{-1} (g(\xx) - f(\xx^*))\,,\]
where 
\[ g(\xx) := f(\xx) + (\beta/2) \left\|\xx\right\|_{\ww,2}^2\]
is the $\ell_2$-regularized version of $f$ with weights given by $\ww$.
Note: The regularized progress statement is true as long as $\xx$ is suboptimal, i.e. $g(\xx) > f(\xx^*)$. Otherwise, we just
have $g(\xx') \leq g(\xx)$.
}
\item {$\xx$ is significantly correlated to the optimal support $S^* := \mathrm{supp}(\xx^*)$:
\[ \left\|\xx_{S^*}\right\|_{\ww^2,2}^2 \geq (4\kappa+6)^{-1} \left\|\xx\right\|_{\ww,2}^2 \,,\]
and
the regularization term restricted to $S^*$ is non-negligible:
\begin{align*}
(\beta/2)\left\|\xx_{S^*}\right\|_{\ww^2,2}^2
\geq (8\kappa+8)^{-1} \left(g(\xx) - f(\xx^*)\right)\,.
\end{align*}}
\end{itemize}
\label{lem:one_step}
\end{lemma}

\paragraph{Comparison to ARHT.} The ARHT algorithm of \cite{axiotis2021sparse} is also able to achieve a sparsity bound of $O(s\kappa)$.
However, their algorithm is not practically desirable for a variety of reasons. 
\begin{itemize}
\item{First of all, it follows the OMP (more accurately, OMP with Removals) paradigm,
which makes local changes to the support of the solution by inserting or removing a single element of the support, and then \emph{fully re-optimizing} the function
on its restriction to this support. Even though the support will generally be very small compared to the ambient dimension $n$, this is still a significant runtime overhead.
In contrast, regularized IHT does not require re-optimization.

Additionally, the fact that in the ARHT only one new element is added at a time leads to an iteration count that scales with $s\kappa$, instead of the $\kappa$ of regularized IHT. This is
a significant speedup, since both algorithms have to evaluate the gradient in each iteration. Therefore, regularized IHT will require $O(s)$ times fewer gradient evaluations.
}
\item{When faced with the non-progress condition, in which the regularized function value does not decrease sufficiently,
ARHT moves by selecting a random index $i$ with probability proportional to $x_i^2$, and proceeds
to \emph{unregularize} this element, i.e. remove it from the sum of regularization terms. Instead, our algorithm is completely deterministic. This is achieved by allowing a
\emph{weighted} regularization term, and 
gradually reducing the regularization weights instead of dropping terms.
}
\item{ARHT requires knowledge of the optimal function value $f(\xx^*)$. The reason is that in each iteration they need to gauge whether enough progress was made in reducing
the value of the regularized function $g$, compared to how far it is from the optimal function value. If so, they would perform the unregularization step.
In contrast, our analysis does not require these two cases (updates to $\xx$ or $\ww$) to be exclusive, and in fact simultaneously updates both, regardless of how much
progress was made in $g$. Thus, our algorithm avoids the expensive overhead of an outer binary search over the optimal value $f(\xx^*)$.
}
\end{itemize}

For all these reasons, as well as its striking simplicity, we believe that regularized IHT can prove to be a useful practical sparse optimization tool.

%% file: lowrank_optimization.tex
\section{Low Rank Optimization Using Regularized Local Search}
\label{sec:lowrank_optimization}

In this section we present a regularized local search algorithm for low rank optimization
of convex functions, that returns
an $\epsilon$-optimal solution with rank 
$O\left(r\left(\kappa+\log\frac{f(\OO)- f(\AA^*)}{\epsilon}\right)\right)$,
where $r$ is the target rank. The algorithm is based on the Local Search algorithm of~\cite{axiotis2021local}, but also uses
adaptive regularization, which leads to a lot new technical hurdles that are addressed in the
analysis.
This is presented in Theorem~\ref{thm:rank_ARHT_plus} and proved
in Appendix~\ref{sec:proofs_lowrank}.

\begin{reptheorem}{thm:rank_ARHT_plus}[Adaptive Regularization for Low Rank Optimization]
Let $f\in\mathbb{R}^{m\times n}\rightarrow\mathbb{R}$ be a
convex function with condition number $\kappa$ and
consider the low rank minimization
problem
\begin{align}
\underset{\rank(\AA) \leq r}{\min}\, f(\AA)\,.
\label{eq:rank_objective_intro}
\end{align}
For any error parameter $\eps>0$,
there exists a polynomial time algorithm that returns a matrix $\AA$ with
$\rank(\AA) \leq O\left(r\left(\kappa+\log \frac{f(\OO)-f(\AA^*)}{\eps}\right)\right)$
and 
$f(\AA) \leq f(\AA^*) + \eps$, where 
$\OO$ is the all-zero matrix and 
$\AA^*$ is any rank-$r$ matrix.
\end{reptheorem}

\paragraph{Discussion about $\eps$ dependence.} 
Some of the technical issues in the rank case have to do with operator \emph{non-commutativity}
and thus pose no issue in the sparsity case. In particular, the extra 
$\log \frac{f(\OO)-f(\AA^*)}{\eps}$ dependence in the rank comes exactly because of these issues.
However, we think that it should be possible to completely remove this dependence
in the future
by a more careful analysis. 

\paragraph{Discussion about computational efficiency.}
We note that the goal of this section is to show an improved rank bound, and not to argue
about the computational efficiency of such an algorithm.
It might be possible to derive an efficient algorithm
by transforming the proof
in Theorem~\ref{thm:rank_ARHT_plus} into a proof for a matrix IHT algorithm, which might be significantly
more efficient, as it will not require solving linear systems in each iteration.
Still, there are a lot of remaining issues to be tackled, as currently the algorithm requires
computing multiple singular value decompositions
and orthogonal projections in each iteration. Therefore working on
a computationally efficient algorithm that can guarantee a rank of $O(r\kappa)$ is a very 
interesting direction for future research.

\paragraph{Matrix regularizer} Getting back into the main ingredients of Theorem~\ref{thm:rank_ARHT_plus}, we describe the choice
of our regularizer. As we are working over general rectangular matrices, we use \emph{two} regularizers,
one for the left singular vectors and one for the right singular vectors of $\AA$. Concretely,
given two weight matrices  $ \YY, \WW$ such that $ \OO \preceq \YY\preceq \II$, $\OO \preceq \WW\preceq \II$, we define
\begin{align*}
\Phi(\AA) = (\beta/4)\left(\langle \WW, \AA\AA^\top\rangle + \langle \YY,\AA^\top\AA\rangle\right)\,,
\end{align*}
where $\beta$ is a bound on the smoothness of $f$.
The gradient of the regularized function is 
\begin{align*}
\nabla g(\AA) = \nabla f(\AA) + (\beta/2)\left(\WW\AA + \AA\YY\right)\,,
\end{align*}
and the new solution $\oAA$ is defined as
\begin{align*}
\oAA = H_{s'-1}\left(\AA\right) - \eta H_1\left(\nabla g(\AA)\right)\,,
\end{align*}
where we remind that the thresholding operator $H_r:\mathbb{R}^{m\times n}\rightarrow\mathbb{R}^{m\times n}$
that is used in the algorithm returns the top $r$ components of the singular value decomposition of a matrix, i.e. given $\MM = \sum\limits_{i=1}^k \lambda_i\uu_i\vv_i^\top$, where
$\vlambda_1\geq\dots\geq \vlambda_k$ are the singular values and $r\leq k$, $H_r(\MM) = \sum\limits_{i=1}^r \lambda_i \uu_i \vv_i^\top$.
In other words, we drop the bottom rank-1 component of $\AA$ and add the top rank-1 component of the gradient.

After taking a step, we re-optimize over matrices with the current left and right singular space, also known as performing a fully corrective step, as in~\cite{shalev2011large,axiotis2021local}.
To do this, we first compute the SVD
$\UU\vSigma\VV^\top$ of $\oAA$ and then solve the optimization problem $\underset{\AA=\UU\XX\VV^\top}{\min}\, g^t(\AA)$.
For simplicity we assume that this optimization problem can be solved exactly, but the analysis can be modified to account for the case when we have an approximate solution
and we are only given a bound on the norm of the gradient (projected onto the relevant subspace), i.e. $\left\|\vPi_{\im(\UU)} \nabla g^t(\AA) \vPi_{\im(\VV)}\right\|_F$.

Whenever there is not enough progress, we make the following updates on the weight matrices $\WW$ and $\YY$:
\begin{align*}
& \WW' = \WW - \WW \AA\AA^\top\WW / \langle \WW,\AA\AA^\top\rangle\\
& \YY' = \YY - \YY \AA^\top\AA\YY / \langle \YY,\AA^\top\AA\rangle\,.
\end{align*}

The full algorithm is in Algorithm~\ref{alg:regularized_local}. 
In the algorithm description we assume that $f(\AA^*)$ is known. This assumption can be
removed by performing binary search over this value, as in~\cite{axiotis2021sparse}. 

\begin{algorithm}%
   \caption{Regularized Local Search}
   \label{alg:regularized_local}
\begin{algorithmic}
\STATE $\AA^0$: initial rank-$r'$ solution
\STATE $\WW^0 = \YY^0 =\II$: initial regularization weights
\STATE $\eta$: step size, $T$: \#iterations
\STATE $c$: weight step size
\FOR{$t=0,\dots, T-1$}
	\STATE $\Phi(\AA) := (\beta/4) \left(\left\langle\WW^t, \AA\AA^\top\right\rangle + \left\langle\YY^t,\AA^\top \AA\right\rangle\right)$
	\STATE $g(\AA) := f(\AA) + \Phi(\AA)$
	\STATE $\oAA = H_{s'-1}\left(\AA^t\right) - 0.5 H_1\left(\eta\nabla g(\AA^t)\right)$
	\STATE $\PP = (\WW^t)^{1/2} \AA^t (\AA^t)^\top (\WW^{t})^{1/2}$
	\STATE $\QQ = (\YY^t)^{1/2} (\AA^t)^\top \AA^t (\YY^t)^{1/2}$
	\STATE $\Delta = g(\AA^t) - f(\AA^*)$
	\STATE 
	\IF {$g(\AA^t) - g(\oAA) \geq (r')^{-1} \Delta$}
		\STATE Let $\UU\vSigma \VV^\top$ be the SVD of $\oAA$
		\STATE $\AA^{t+1} = \underset{\AA=\UU\XX\VV^\top}{\mathrm{argmin}}\, g(\AA)$
		\STATE $\WW^{t+1}, \YY^{t+1} = \WW^t, \YY^t$
	\STATE
	\ELSIF {
		\ifdefined\arxiv
		$\max\left\{\mathrm{Tr}\left[H_{r}\left(\PP\right)\right], \mathrm{Tr}\left[H_{r}\left(\QQ\right)\right]\right\} \geq (0.4/\beta) \Delta $
		\else
		\[ 
		\max\left\{\mathrm{Tr}\left[H_{r}\left(\PP\right)\right], \mathrm{Tr}\left[H_{r}\left(\QQ\right)\right]\right\} \geq (0.4/\beta) \Delta 
		\]
		\fi
	}
		\STATE $\AA^{t+1} = \AA^t$
		\STATE $\WW^{t+1} = (\WW^t)^{1/2}\left(\II - r^{-1}\vPi_{\im(\PP)}\right)(\WW^t)^{1/2}$
		\STATE $\YY^{t+1} = (\YY^t)^{1/2}\left(\II - r^{-1}\vPi_{\im(\QQ)}\right)(\YY^t)^{1/2}$
	\STATE
	\ELSE
		\STATE $\AA^{t+1} = \AA^t$
		\STATE $\WW^{t+1} = \WW^t - \frac{\WW^t \AA^t(\AA^t)^\top\WW^t}{\langle \WW^t,\AA^t(\AA^t)^\top\rangle}$
		\STATE $\YY^{t+1} = \YY^t - \frac{\YY^t (\AA^t)^\top\AA^t\YY^t}{\langle \YY^t,(\AA^t)^\top\AA^t\rangle}$
	\ENDIF
\ENDFOR
\end{algorithmic}
\end{algorithm}

%% file: experiments.tex
\section{Experiments}

\paragraph{Introduction.} In this section we present numerical experiments in order to compare the performance of IHT and regularized IHT (Algorithm~\ref{alg:regularized_iht}) in training sparse linear models.
In particular, we will look at the tasks of linear regression and logistic regression using both real and synthetic data.
In the former, we are given a matrix $\AA\in\mathbb{R}^{m\times n}$, where each row represents an example and each column a feature, and a vector $\bb\in\mathbb{R}^m$ that represents the ground truth
outputs, and our objective is to minimize the $\ell_2$ loss
\begin{align*}
f(\xx) = (1/2) \left\|\AA\xx-\bb\right\|_2^2\,.
\end{align*}
In logistic regression, $\bb$ has binary instead of real entries, and our objective is to minimize the logistic loss
\begin{align*}
f(\xx) = -\sum\limits_{i=1}^m \left(b_i\log\sigma(\AA\xx)_i + (1-b_i) \log (1-\sigma(\AA\xx)_i)\right)\,,
\end{align*}
where $\sigma(z) = \left(1 + e^{-z}\right)^{-1}$ is the sigmoid function. As is common, we look at the \emph{regularized} logistic regression objective:
\begin{align*}
f(\xx) = -\sum\limits_{i=1}^m \left(b_i\log\sigma(\AA\xx)_i + (1-b_i) \log (1-\sigma(\AA\xx)_i)\right) + (\rho/2) \left\|\xx\right\|_2^2\,,
\end{align*}
for some $\rho > 0$. For our experiments we use $\rho=0.1$.

\paragraph{Preprocessing and choice of parameters.}

The only preprocessing we perform is to center the columns of $\AA$, i.e. we subtract the mean of each column from each entry of the column,
and then scale the columns to unit $\ell_2$ norm. 
This ensures that for any sparsity parameter $s'\in[n]$, the function $f$ is $s'$-smooth when restricted to $s'$-sparse directions, or in other words the 
$s'$-restricted smoothness constant of $f$ is at most $s'$.
Thus we set our smoothness estimate to $\beta := s'$. Our smoothness estimate $\beta$ influences the (regularized) IHT algorithm in two ways. First, as the step size of the algorithm
is given by $1/\beta$, a value of $\beta$ that is too large can slow down the algorithm, or even get it stuck to a local minimum. Second, the strength of the 
regularization term in regularized IHT should be close to the $(s+s')$-restricted smoothness constant, as shown in the analysis of Theorem~\ref{thm:regularized_iht}.

Even though having a perfectly accurate estimate of the smoothness constant is not necessary, a more accurate estimate improves the performance of the algorithm. In fact,
the estimate $1/s'$ for the step size is generally too conservative. When used in practice, one should either tune this parameter or use a variable/adaptive step size to achieve the best results.

For the weight step size of regularized IHT, we set the weight step size to $c=s'/T$, but we also experiment with how changing $c$ affects the performance of the algorithm.
The downside of this setting is that it requires knowing the number of iterations a priori. However, in practice one could tune $c$ and then run the algorithm for $O(s'/c)$ iterations.
Note that ideally, based on the theoretical analysis, $T$ would be proportional to the restricted condition number of $f$, however this quantity is hard to compute in general. Another idea to
avoid this in practice could be to let $c$ be a variable step size.

\paragraph{Implementation.} Both the IHT and regularized IHT algorithms are incredibly simple, and can
be described in a few lines of python code, as can be seen in Figure~\ref{fig:python}.
Note that in comparison to Algorithm~\ref{alg:regularized_iht} we do not perform the conditional
assignment. All the experiments were run on a single 2.6GHz Intel Core i7 core of a
2019 MacBook Pro with 16GB DDR4 RAM using Python 3.9.10.

\begin{figure}
\begin{python}
import numpy as np

def IHT(n, s):
	x = np.zeros(n)
	for _ in range(T):
		x_new = x - eta * grad(x)
		x_new[np.argsort(np.abs(x_new))[:-s]] = 0
		x = x_new
	return x

def RegIHT(n, s):
	x, w = np.zeros(n), np.ones(n)
	for _ in range(T):
        x_new = (1 - 0.5 * w) * x - 0.5 * eta * grad(x)
        x_new[np.argsort(np.abs(x_new))[:-s]] = 0
        reg = np.sum(w * x**2)
        if reg != 0:
            w = w * (1 - c * w * x**2 / reg)
            w[w <= round_th] = 0
        x = x_new
	return x
\end{python}
\caption{Our python implementations of IHT, RegIHT, where grad is the gradient function,
n is the total number of features, s is the desired sparsity level,
eta is the step size, c is the weight step size, and round\_th is the weight rounding threshold,
which we set to $0.5$. Note that 
grad(x) = np.dot(A.T, np.dot(A, x) - b) for linear regression
and
grad(x) = np.dot(A.T, expit(np.dot(A,x)) - b) for logistic regression,
where expit is the sigmoid function.
}
\label{fig:python}
\end{figure}

\subsection{Real data}

\begin{figure}
\vskip 0.2in
\begin{center}
	\includegraphics[width=\ifdefined\arxiv 0.49\fi\columnwidth]{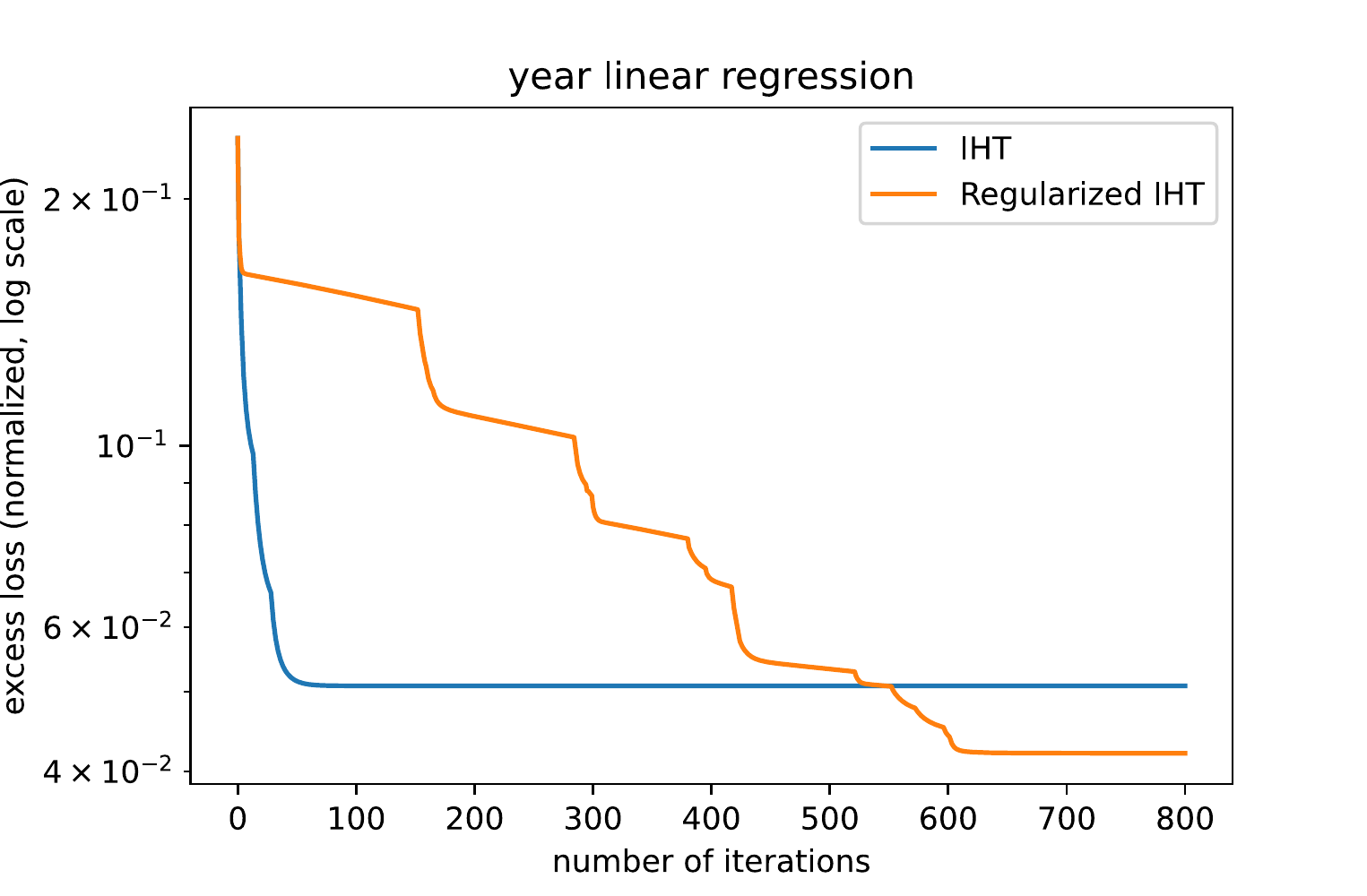} 
	\includegraphics[width=\ifdefined\arxiv 0.49\fi\columnwidth]{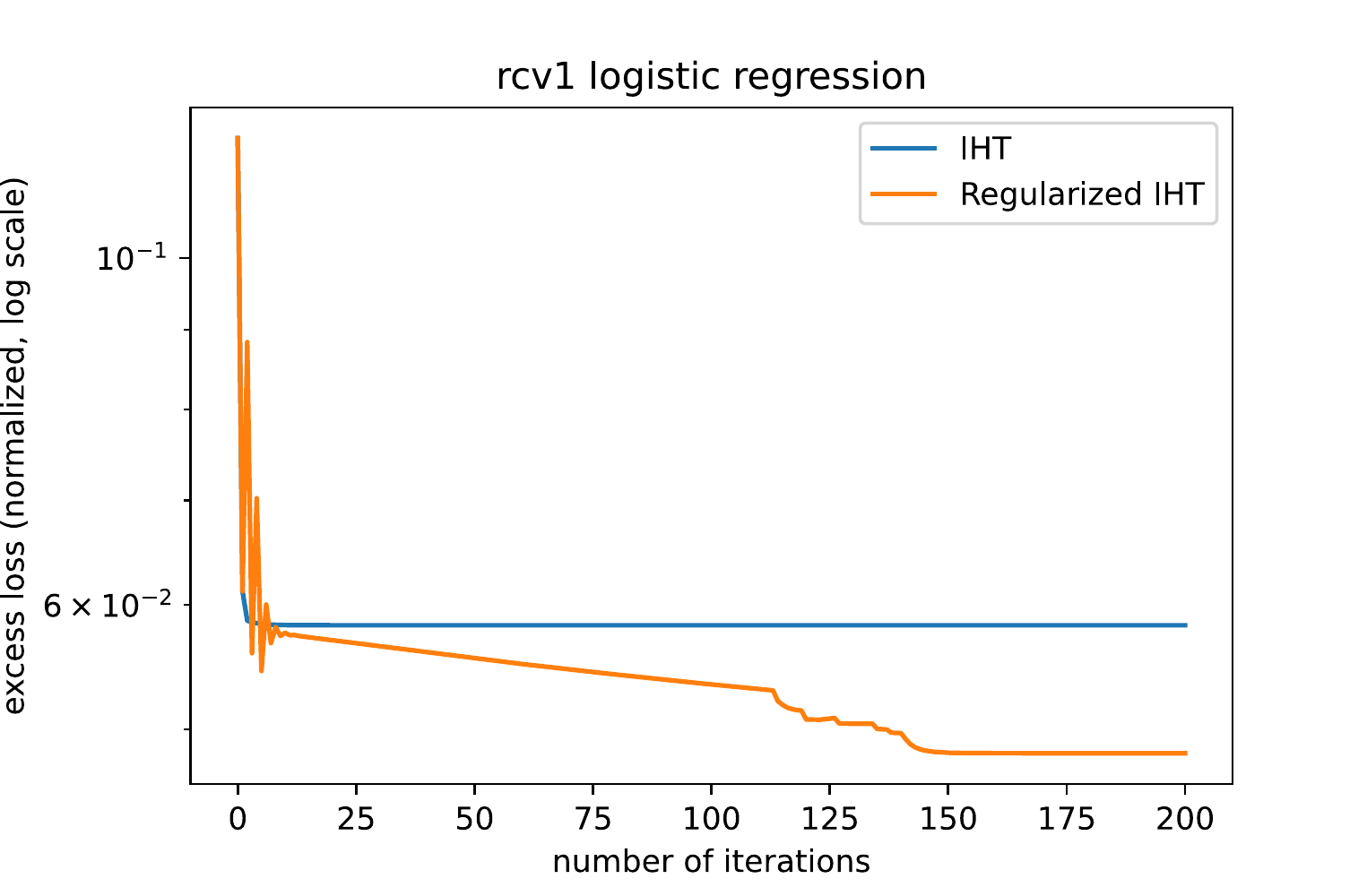} 
\end{center}
\caption{IHT vs Regularized IHT performance on the year and rcv1 datasets with fixed sparsity levels
$s=11$ and $s=10$ respectively.
On the $x$ axis we have number of iterations and on the $y$ axis we have the normalized excess loss (compared to the dense global optimum), 
in log scale. The excess loss of regularized IHT is less than that of IHT, specifically 17.3\% and 17.2\% respectively less in the two experiments.
We can see that, initially, regularized IHT has a much higher error than IHT, but after some iterations it finds a solution with lower error than the one found by IHT.
This phenomenon of regularized IHT having decreased performance in the first part of the algorithm is to be expected, because the algorithm runs on a regularized function, and so tries to keep
not just $f(\xx)$ but also $\left\|\xx\right\|_2^2$ small. After some iterations of this, however, the situation changes. The algorithm has learnt regularization weights that 
are closer to the optimal ones, and can thus converge to sparser
solutions than IHT (equivalently, lower error solutions with the same sparsity, which is what is shown in this plot). 
}
\label{fig12}
\end{figure}

We first experiment with real data, specifically the \emph{year} regression dataset 
from UCI~\cite{Dua2019} and the \emph{rcv1} binary classification dataset~\cite{lewis2004rcv1}.
In Figure~\ref{fig12} we have a comparison between the error of the solution returned by IHT 
and regularized IHT for a fixed sparsity level. Specifically, if we let $\xx^{**}$ be the (dense) global minimizer of $f$,
we plot the logarithm of the (normalized) \emph{excess loss} $(f(\xx) - f(\xx^{**})) / f(\zerov)$ against the number of iterations.
Note that $f(\xx^{**})$ will typically be considerably lower than the loss of the sparse optimum $f(\xx^*)$.
In order to make a fair comparison, for each algorithm we pick the best fixed step size of the form $2^i/s$ for integer $i\geq 0$, where $s$ is the fixed sparsity level. The best
step sizes of IHT and regularized IHT end up being $2/s, 4/s$ respectively for the linear regression example,
and $8/s, 16/s$ respectively for the logistic regression example.

\begin{figure}
\vskip 0.2in
\begin{center}
	\includegraphics[width=\ifdefined\arxiv 0.49\fi\columnwidth]{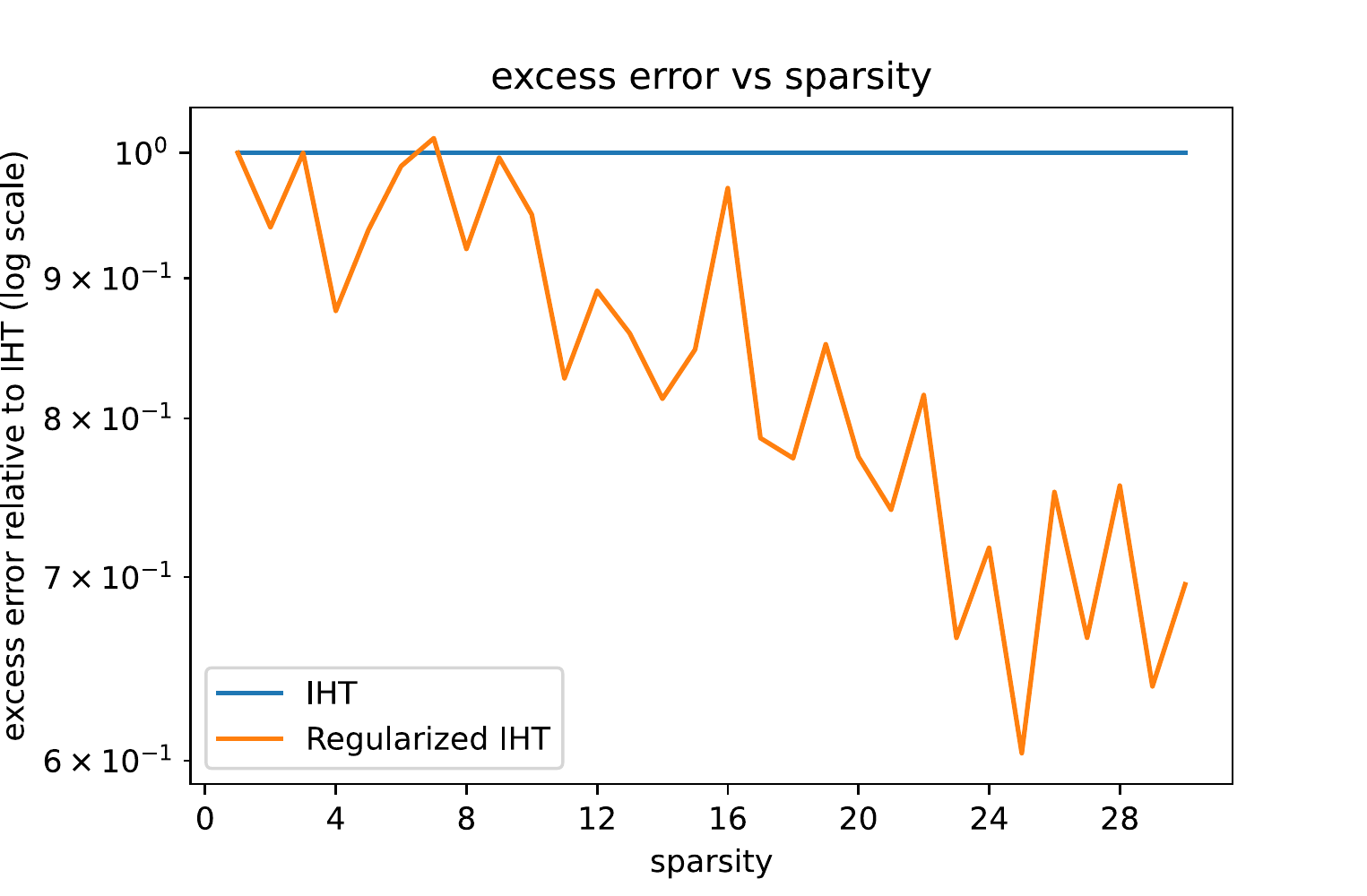} 
	\includegraphics[width=\ifdefined\arxiv 0.49\fi\columnwidth]{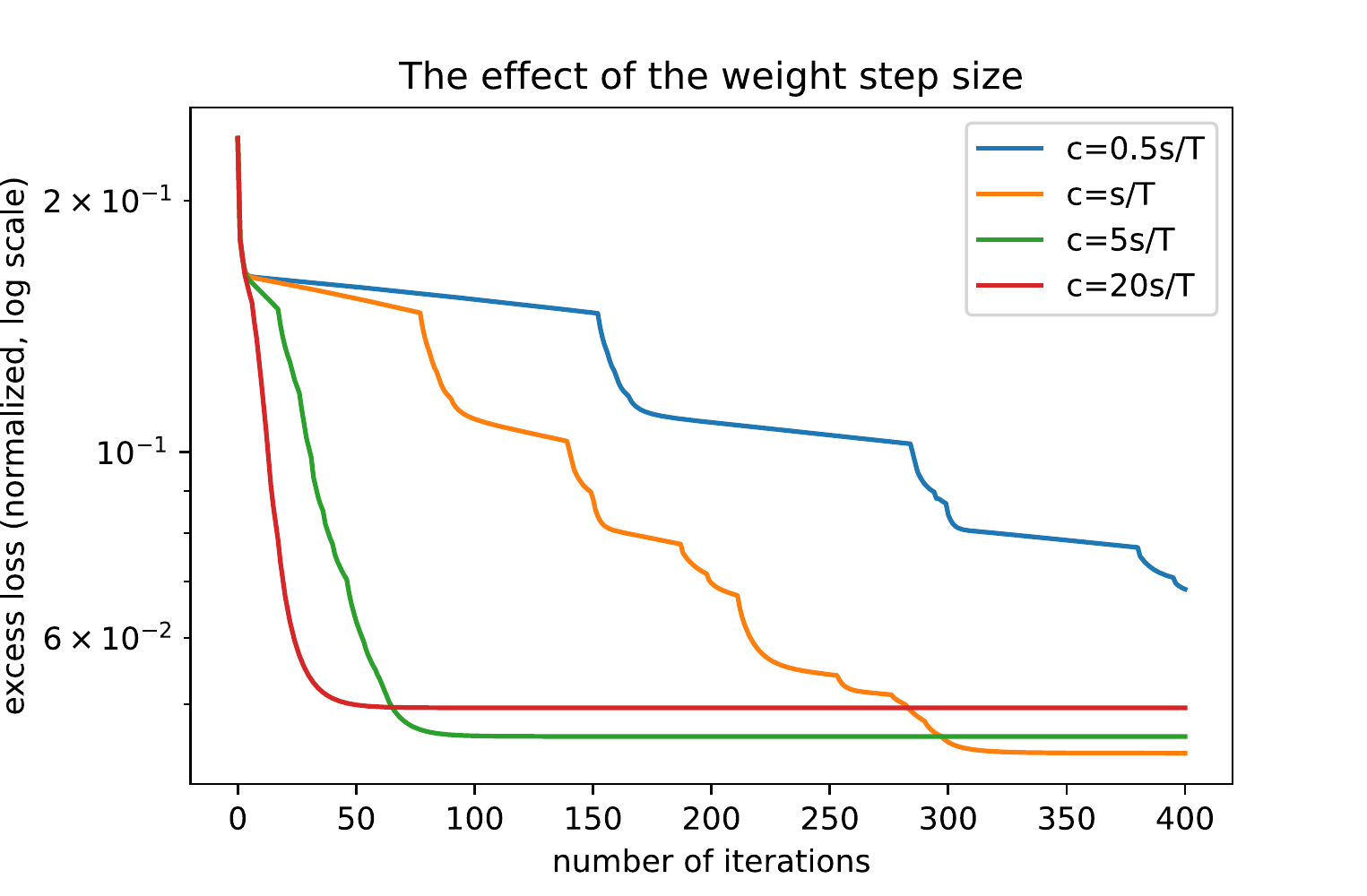} 
\end{center}
\caption{
\emph{Left:} Excess error of regularized IHT relative to IHT in the \emph{year} dataset, where sparsity values range 
from $1$ to $30$. Both algorithms are run for $T=800$ iterations.\\
	\emph{Right:} Error rate vs number of iterations of regularized IHT on the \emph{year} dataset with fixed sparsity $s=11$ and step size $\eta=4/s$,
using different values for the weight step size $c$. Here we can see an interesting tradeoff between the number of iterations
and the error of the solution that is eventually returned. In particular, the \emph{larger} $c$ is, the \emph{faster} the degradation of regularization weights. Thus, for $c\rightarrow\infty$,
the algorithm tends to be the same as IHT. On the other hand, with \emph{smaller} values of $c$, one can get an improved error rate, but at the cost of a larger number of iterations. This is
because the regularization weights decrease slowly, and so in the early 
iterations of the algorithm (i.e. until the proper weights are learned), 
the regularization term will account for a significant fraction of the 
objective function value.
}
\label{fig4}
\end{figure}
In Figure~\ref{fig4} left we compare IHT and regularized IHT for different sparsity levels on the year dataset.
If $e_1$ and $e_2$ are the excess errors of IHT and regularized IHT respectively, we plot $e_2 / e_1$, which is the relative excess error
of regularized IHT with respect to that of IHT. We notice a reduction of up to $40\%$ on the excess error.
In Figure~\ref{fig4} right we examine the effect of the choice of the weight step size $c$.
We conclude that $c$ can give a tradeoff between runtime and accuracy, as setting it to a large value
will lead to faster weight decay and thus resemble IHT, while a small value of $c$ will lead to slow
weight decrease, which will lead to more iterations but also potentially recover an improved solution.

\subsection{Synthetic data}
We now turn to synthetically generated linear regression instances. The first 
result presented in Figure~\ref{fig13}
is the hard 
IHT instance that we derived in our lower bound in Appendix~\ref{sec:lower_bounds}. This experiment
shows that there exist examples where, with bad initialization, IHT cannot decrease the objective 
at all (i.e. is stuck at a local minimum), while regularized IHT with the same initialization manages
to reduce the loss by more than 70\%.

The second result is a result in the well known setting of sparse signal recovery from 
linear measurements.
We generate a matrix $\AA$ with  entries that are sampled i.i.d. from the standard normal distribution, an $s$-sparse signal $\xx$ again with
  entries sampled i.i.d. from the standard normal distribution, and an observed vector $\bb := \AA \xx$. The goal is to recover $\xx$ by
minimizing the objective 
\[ f(\xx) = (1/2) \left\|\AA\xx-\bb\right\|_2^2 \,.\]
In Figure~\ref{fig13}, we plot the normalized value of this objective, after running both IHT and
regularized IHT for the same number of iterations. 
Here we pick the best step size per instance, 
starting from $\eta=1/s$ and increasing in multiples of $1.2$. Also, for each fixed value of 
$s$ and algorithm, we run the experiments $20$ times in order to account for the variance.
The results show a superiority in the performance of regularized IHT for the sparse signal recovery
task.
\begin{figure}
\vskip 0.2in
\begin{center}
	\includegraphics[width=\ifdefined\arxiv 0.49\fi\columnwidth]{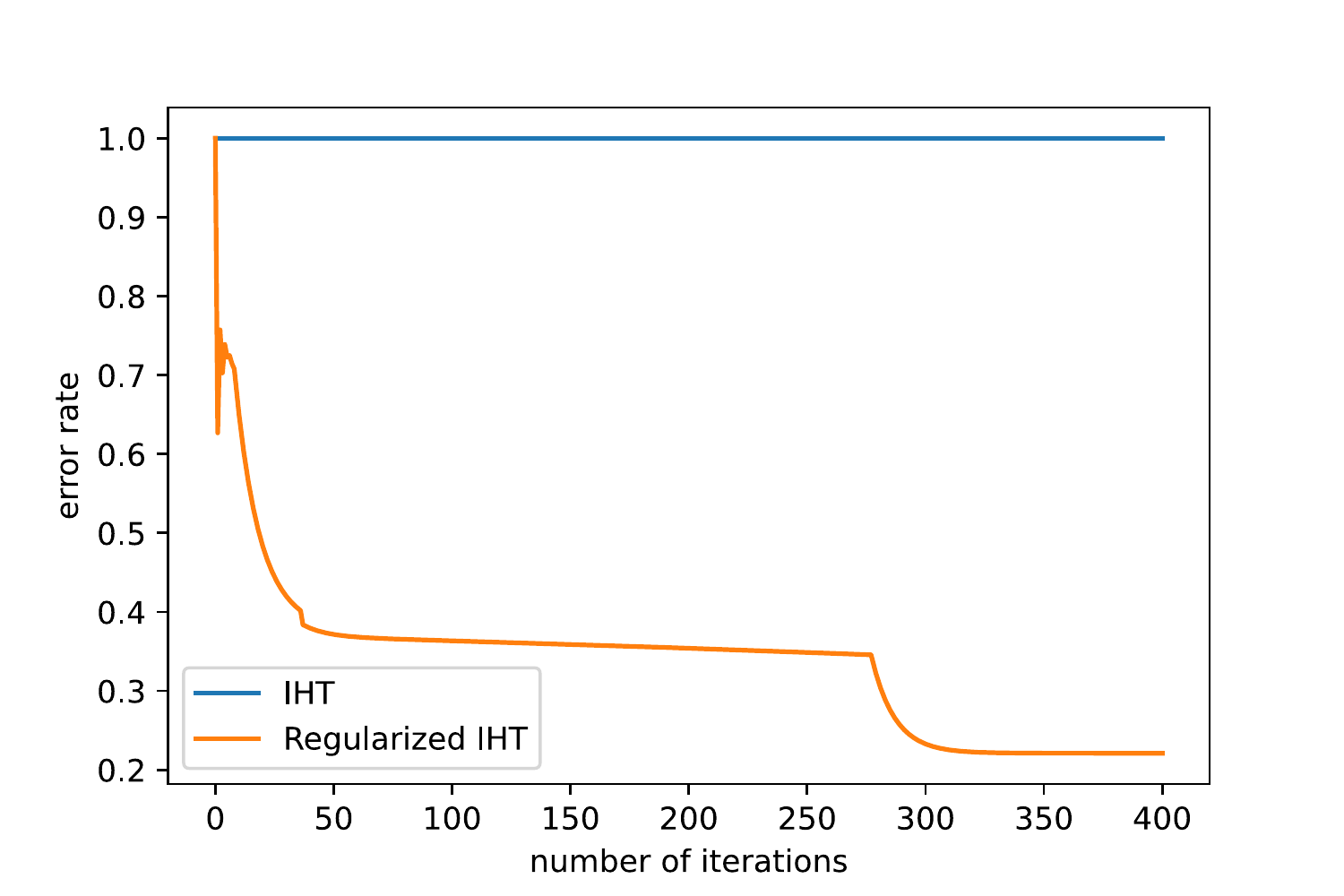}
	\includegraphics[width=\ifdefined\arxiv 0.49\fi\columnwidth]{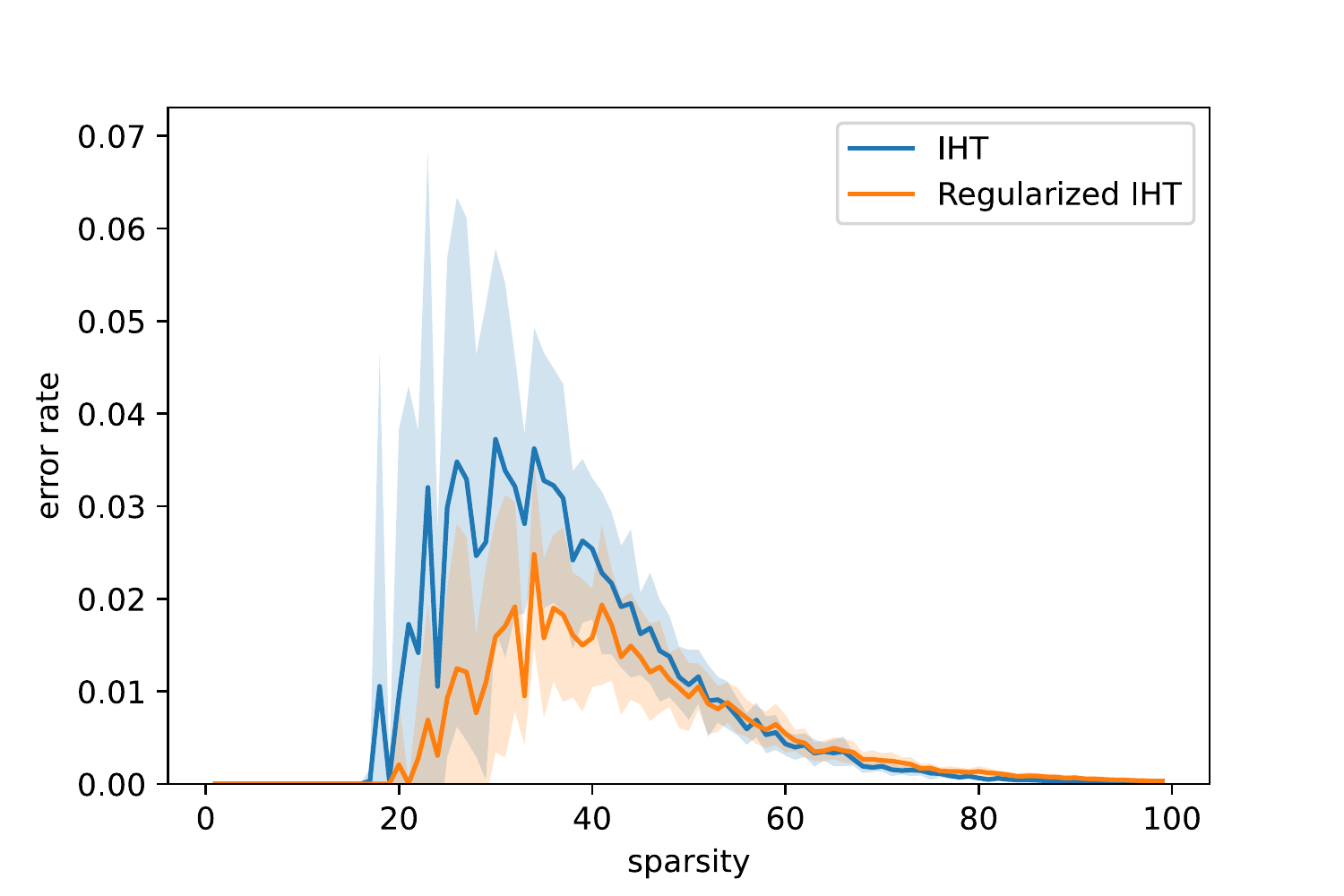} 
\end{center}
\caption{
\emph{Left:} A demonstration of a ~80\% decrease in loss by using regularized IHT instead of IHT on the hard instance for IHT presented in Section~\ref{sec:lower_bounds}.
We have generated the data with a condition number of $\kappa=20$, and a planted sparse solution with sparsity $s=2$. The dimension is $n=842$.
It can be observed that, for the given initialization vector, IHT never makes any progress on decreasing the error. In contrast, regularized IHT is able to decrease it by almost a factor of $5$.
\emph{Right:} Sparse signal recovery, where $\AA$ is an $100\times 800$ measurement matrix,
the sparsity level ranges from $1$ to $100$, and each algorithm is run for $240$ iterations. Bands of
1 standard error are shown, after running each data point $20$ times independently.
}
\label{fig13}
\end{figure}

%% file: appendix.tex
\section{Proof of Theorem~\ref{thm:regularized_iht}}
\label{sec:proof_regularized_iht}
\begin{reptheorem}{thm:regularized_iht}[Regularized IHT]
Let $f\in\mathbb{R}^{n}\rightarrow\mathbb{R}$ be a convex function that is $\beta$-smooth
and $\alpha$-strongly convex, with
condition number $\kappa=\beta/\alpha$, and $\xx^*$
be an (unknown) $s$-sparse solution with support $S^*$. Then,
running 
Algorithm~\ref{alg:regularized_iht} with $\eta = (2\beta)^{-1}$ and $c=s'/(4T)$
for 
\[ T = O\left(\kappa \log \frac{f(\xx^0) + (\beta/2)\left\|\xx^0\right\|_2^2 - f(\xx^*)}{\eps}\right) \]
iterations starting from an arbitrary $s'=O(s\kappa)$-sparse solution $\xx^0$, 
the algorithm returns an $s'$-sparse solution $\xx^T$ such that
$f(\xx^T) \leq f(\xx^*) + \eps$.
Furthermore, each iteration requires
$O(1)$ evaluations of $f$, $\nabla f$, and $O(n)$ additional time.
\end{reptheorem}
\begin{proof}
We repeatedly apply Lemma~\ref{lem:one_step} for 
\[ T=64(\kappa+1)\log\frac{f(\xx^0) + (\beta/2)\left\|\xx^0\right\|_2^2 - f(\xx^*)}{\eps}\]
iterations.
We define the regularized function 
\[ g^t(\xx) := f(\xx) + (\beta/2)\left\|\xx\right\|_{\ww^t,2}^2 \,, \]
where $\ww^t$ are the weights before iteration $t\in[0,T-1]$.
Specifically, for each $t$ we apply Lemma~\ref{lem:one_step}
on the current solution $\xx^{t}$ and obtain the solution $\xx^{t+1}$.

Before moving forward, we give an intuitive summary of the proof and the role of Lemma~\ref{lem:one_step}.
As long as IHT makes ``sufficient'' progress on the regularized function $g^t$, this is satisfactory for 
the original function $f$ as well, because $f(\xx) \leq g^t(\xx)$ for all $\xx$.
This is the case of the first bullet of Lemma~\ref{lem:one_step}. If it stops making sufficient progress,
this means we are at an (approximate) sparse optimum for $g^t$,
although it is not necessarily a good sparse solution for $f$, which is the objective we are aiming to minimize.
This is where the second bullet of Lemma~\ref{lem:one_step} comes in, which gives necessary conditions
for the above non-progress phenomenon (in other words, a partial characterization of the local minima
encountered when running IHT on a regularized function).
Specifically, the following condition is central to our approach:
\[ \left\|\xx_{S^*}^t\right\|_{(\ww^t)^2,2}^2 \geq (4\kappa+6)^{-1} \left\|\xx^t\right\|_{\ww^t,2}^2 \,.\]
We use this condition in the second part of the 
proof (after Case 2) to motivate a weight update from $\ww^t$ to $\ww^{t+1}$, and show
that, exactly because of this condition, a lot of the weight decrease is concentrated inside the 
optimal support $S^*$. As the total weight decrease in $S^*$ is bounded by $s$, this gives a bound
on the total number of iterations with insufficient decrease of $g^t$.
If not for this condition, we would not be able to bound the number of
such iterations and would have potentially remained forever stuck at a local minimum.

Now we are ready to move to the technical proof.
In order to make sure that $g^{t+1}(\xx^{t+1}) \leq g^{t+1}(\xx^{t})$, we revert
to the previous solution if the one returned by Lemma~\ref{lem:one_step} has a larger value of $g^{t+1}$. This is
exactly what the conditional in Algorithm~\ref{alg:regularized_iht} is for. The property that 
$g^{t+1}(\xx^{t+1}) \leq g^{t+1}(\xx^{t})$ is only used in the very last part of the proof.

Let us assume that $g^{t}(\xx^{t}) > f(\xx^*)$ at all times, as otherwise the statement holds by the fact that $g^{t}(\xx^{t})$
is non-increasing as a function of $t$ and upper bounds $f(\xx^t)$ for all $t$. We have $g^{t+1}(\xx^{t+1}) \leq g^t(\xx^{t})$
by the fact that $\ww^{t+1} \leq \ww^{t}$ and $g^t(\xx^{t+1}) \leq g^{t}(\xx^t)$ by the guarantees of Lemma~\ref{lem:one_step}.

If the first bullet of Lemma~\ref{lem:one_step} holds, we have that the value of 
$g$ decreases considerably on iteration $t$, i.e.
\begin{align*}
& g^{t+1}(\xx^{t+1}) \\
& \leq g^{t}(\xx^{t+1})\\
& \leq g^{t}(\xx^{t}) - (16\kappa)^{-1} (g^{t}(\xx^{t}) - f(\xx^*))\,.
\end{align*}
Let us call these iterations \emph{progress iterations}, and the other ones (where the second bullet of Lemma~\ref{lem:one_step} holds) \emph{weight iterations}.
Now, since $g^t(\xx^t)$ is non-increasing as a function of $t$,
after $16\kappa \log \frac{g^0(\xx^0) - f(\xx^*)}{\eps}$ progress iterations 
we will have 
\[ f(\xx^t) \leq g^t(\xx^t) \leq f(\xx^*) + \eps \,, \]
and so we will be done. From now on let us assume this is not the case, so there are at least
\[ T - 16\kappa \log \frac{g^0(\xx^0) - f(\xx^*)}{\eps} \geq 3T/4\]
weight iterations.

We remind that in each weight iteration, we have
\begin{align}
& \left\|\xx_{S^*}^t\right\|_{(\ww^t)^2,2}^2
\geq 
(4\kappa + 6)^{-1} \left\|\xx^t\right\|_{\ww^{t},2}^2
\label{eq:intersection_vs_all}\\
& 
(\beta/2)\left\|\xx_{S^*}^t\right\|_{(\ww^{t})^2,2}^2
\geq (8\kappa+8)^{-1} \left(g^{t}(\xx^t) - f(\xx^*)\right)
\label{eq:intersection_vs_function}
\,.
\end{align}
In words, (\ref{eq:intersection_vs_all}) roughly implies
that at least an $\Omega(1/\kappa)$ fraction of the
mass of $(\ww^{t}\xx^t)^2$ lies inside $S^*$.
Therefore, if we decrease $\ww^{t}$ by a quantity proportional to $(\ww^{t} \xx^t)^2$, the total sum of weights
will decrease at most $O(\kappa)$ times faster than the sum of weights inside $S^*$. As the latter
quantity can only decrease by $s$ overall, the total decrease of weights will be $O(s \kappa)$.

Concretely, after each iteration we update the 
regularization weights as follows:
\begin{align*}
\ww^{t+1} = \left(\ww^{t} - c\cdot (\ww^{t} \xx^t)^2 / \left\|\xx^t\right\|_{\ww^{t},2}^2\right)_{\geq 1/2}\,,
\end{align*}
for some $c > 0$ to be determined later.
First of all, note that the weights are non-increasing. Now, if not for the thresholding operation, it is easy to see
that the total weight decrease is at most $c$. The thresholding operation can only double this weight decrease to $2c$. Concretely, for all $t$ we define a vector $\oww^t$ such that
\begin{align*}
\ow_i^t = \begin{cases}
w_i^t & \text{if $w_i^t \geq 1/2$}\\
1/2 & \text{if $w_i^t = 0$}\,.
\end{cases}
\end{align*}
Clearly, $\frac{1}{2} \left(\onev - \ww^t\right) \leq \onev - \oww^t \leq \onev - \ww^t$. Now, we have
\begin{align*}
&  \left\|\oww^{t}\right\|_1-\left\|\oww^{t+1}\right\|_1\\
&  \leq  c \left\|\xx^t\right\|_{(\ww^{t})^2,2}^2 / \left\|\xx^t\right\|_{\ww^{t},2}^2\\
&  \leq  c\,,
\end{align*}
and, summing up for all $t$ we get 
\[ \left\|\onev - \oww^T\right\|_1 = \left\|\oww^0\right\|_1 - \left\|\oww^T\right\|_1 \leq cT \,. \]
Therefore, 
\begin{align*}
\left\|\onev - \ww^T\right\|_1
\leq 2 \left\|\onev - \oww^T\right\|_1
\leq 2cT\,, 
\end{align*}
and so $\left\|\ww^T\right\|_1 \geq n - 2cT$.

Therefore, the condition $\left\|\ww^t\right\|_1 \geq n - s'/2$ of Lemma~\ref{lem:one_step} is satisfied
for all $t$ as long as $c \leq s'/(4T)$.
In order to bound the number of iterations, we distinguish two cases for the sum of weights inside $S^*$.
\paragraph{Case 1:} The sum of weights inside $S^*$ decreases by $\geq 4 s / T$.

This case cannot happen more than $T/4$
times since the sum of weights
inside $S^*$ can only decrease by $s$ in total. Therefore, case 2 below happens at least $T/2$ times.

\paragraph{Case 2:} The sum of weights inside $S^*$ decreases by $<4s/T$.

Note that the decrease in the sum of weights in $S^*$ is exactly equal to
\begin{align*}
\sum\limits_{i\in S^*} \begin{cases}
 c\cdot (w_i^{t} x_i^t)^2 / \left\|\xx^t\right\|_{\ww^{t},2}^2 & \text{if this is $\leq w_i^t - 1/2$}\\
 w_i^t & \text{otherwise}\,.
\end{cases}
\end{align*}
Let $T^*$ be the set of indices $i\in S^*$ for which the second case is true, i.e.
\[ c\cdot (w_i^{t} x_i^t)^2 / \left\|\xx^t\right\|_{\ww^{t},2}^2 >  w_i^t - 1/2\,.\]
The total weight decrease from elements in $S^*\backslash T^*$ is then 
\begin{align*}
& \sum\limits_{i\in S^*\backslash T^*}
c\cdot (w_i^{t} x_i^t)^2 / \left\|\xx^t\right\|_{\ww^{t},2}^2\\
& = c \left\|\xx_{S^*\backslash T^*}^t\right\|_{(\ww^{t})^2,2}^2 / \left\|\xx^t\right\|_{\ww^{t},2}^2\\
& \geq 
\frac{c}{4\kappa + 6} \left\|\xx_{S^*\backslash T^*}^t\right\|_{(\ww^{t})^2,2}^2 / 
\left\|\xx_{S^*}^t\right\|_{(\ww^{t})^2,2}^2\,,
\end{align*}
where we used (\ref{eq:intersection_vs_all}). 
As we have assumed that this decrease is less than
$4s/T$, we have that 
\begin{equation}
\begin{aligned}
\left\|\xx_{T^*}^t\right\|_{(\ww^{t})^2,2}^2 
& =
\left\|\xx_{S^*}^t\right\|_{(\ww^{t})^2,2}^2 
- \left\|\xx_{S^*\backslash T^*}^t\right\|_{(\ww^{t})^2,2}^2 
\\
& \geq \left(1 - \frac{4s(4\kappa+6)}{cT}\right) \left\|\xx_{S^*}^t\right\|_{(\ww^{t})^2,2}^2 \\
& \geq (1/2) \left\|\xx_{S^*}^t\right\|_{(\ww^{t})^2,2}^2 \,,
\end{aligned}
\label{eq:T_star_large}
\end{equation}
as long as $c \geq 8s(4\kappa+6) / T$. We can pick such a $c$ as long as 
\begin{align*}
8s(4\kappa+6) / T \leq c \leq s'/(4T) \Leftrightarrow s' \geq 32(4\kappa+6) s\,.
\end{align*}
Now, to deal with the fact that the sum weights in
$T^*$ might not decrease sufficiently,
note that all the weights in $T^*$ are being set to $0$, i.e. $w_i^{t+1} = 0$ for all $i\in T^*$.
Together with (\ref{eq:T_star_large}) and (\ref{eq:intersection_vs_function})
this means that we can make significant progress in function value. To see this, note that
\begin{align*}
& g^{t+1}(\xx^{t+1})\\
& \leq g^{t+1}(\xx^t) \\
& \leq g^{t}(\xx^{t}) - (\beta/2) \left\|\xx_{T^*}^{t}\right\|_{\ww^{t},2}^2\\
& \leq g^{t}(\xx^{t}) - (\beta/2) \left\|\xx_{T^*}^t\right\|_{(\ww^{t})^2,2}^2\\
& \leq g^{t}(\xx^{t}) - (\beta/4) \left\|\xx_{S^*}^t\right\|_{(\ww^{t})^2,2}^2\\
& \leq g^{t}(\xx^{t}) - (16\kappa+16)^{-1} \left(g^{t}(\xx^t) - f(\xx^*)\right)\,,
\end{align*}
which can happen at most 
\begin{align*}
16(\kappa+1) \log \frac{g^0(\xx^0) - f(\xx^*)}{\eps} \leq T / 4
\end{align*}
times.

\ifx 0 
For the sum of weights inside $S^*$ we have 
\begin{align*}
& \left\|\ww_{S^*}^t\right\|_1 - \left\|\ww_{S^*}^{t-1}\right\|_1\\
& = - c \left\|\xx_{S^*}^t\right\|_{\ww^{t-1},2}^2 / \left\|\xx^t\right\|_{\ww^{t-1},2}^2\\
& \leq - c / (8\kappa + 1)\,.
\end{align*}

Overall, we have 
\begin{align*}
-s^*
& \leq \left\|\ww_{S^*}^T\right\|_1 - \left\|\ww_{S^*}^{0}\right\|_1\\
& = \sum\limits_{t=1}^T \left(\left\|\ww_{S^*}^t\right\|_1 - \left\|\ww_{S^*}^{t-1}\right\|_1\right)\\
& \leq \sum\limits_{t\in I_2} \left(\left\|\ww_{S^*}^t\right\|_1 - \left\|\ww_{S^*}^{t-1}\right\|_1\right)\\
& \leq -|I_2| \cdot c / (8\kappa + 1)\,,
\end{align*}
implying that 
\[ |I_2| \leq (8\kappa + 1) s^* / c \leq 4(8\kappa + 1) T s^* / s \leq T / 2 \,,\]
as long as $s \geq 36 \kappa s^*$. So we are done.
\fi

\end{proof}

\section{Proof of Lemma~\ref{lem:one_step}}
\label{sec:proof_lem_one_step}

\begin{replemma}{lem:one_step}[Regularized IHT step progress]
\it
Let $f\in\mathbb{R}^{n}\rightarrow\mathbb{R}$ be a convex 
function that is $\beta$-smooth and $\alpha$-strongly convex, 
$\kappa=\beta/\alpha$ be its condition number, and $\xx^*$ be any $s$-sparse
solution.

Given any $s'$-sparse solution $\xx\in\mathbb{R}^n$ where
\begin{align*}
s' \geq (128\kappa+2)s
\end{align*}
and a weight vector $\ww\in\left(\{0\}\cup[1/2,1]\right)^n$ 
such that $\left\|\ww\right\|_1 \geq n - s'/2$, we make the following update:
\[ \xx' = H_{s'}\left((\onev - 0.5 \ww) \xx - (2\beta)^{-1} \nabla f(\xx)\right) \,. \]
Then, at least one of the following two conditions holds:
\begin{itemize}
\item {
Updating $\xx$ makes regularized progress:
\[ g(\xx') \leq g(\xx) - (16\kappa)^{-1} (g(\xx) - f(\xx^*))\,,\]
where 
\[ g(\xx) := f(\xx) + (\beta/2) \left\|\xx\right\|_{\ww,2}^2\]
is the $\ell_2$-regularized version of $f$ with weights given by $\ww$.
Note: The regularized progress statement is true as long as $\xx$ is suboptimal, i.e. $g(\xx) > f(\xx^*)$. Otherwise, we just
have $g(\xx') \leq g(\xx)$.
}
\item {$\xx$ is significantly correlated to the optimal support $S^* := \mathrm{supp}(\xx^*)$:
\[ \left\|\xx_{S^*}\right\|_{\ww^2,2}^2 \geq (4\kappa+6)^{-1} \left\|\xx\right\|_{\ww,2}^2 \,,\]
and
the regularization term restricted to $S^*$ is non-negligible:
\begin{align*}
(\beta/2)\left\|\xx_{S^*}\right\|_{\ww^2,2}^2
\geq (8\kappa+8)^{-1} \left(g(\xx) - f(\xx^*)\right)\,.
\end{align*}}
\end{itemize}
\end{replemma}
\begin{proof}
By using the fact that $f$ is $\beta$-smooth,
and so $g$ is $2\beta$-smooth due to $\ww\leq \onev$,
for any $\xx'$ we obtain
\begin{align}
 g(\xx') - g(\xx)
 \leq \left\langle \nabla g(\xx), \xx'-\xx\right\rangle
+ \beta \left\|\xx' - \xx\right\|_{2}^2\,.
\label{eq:g_smoothness}
\end{align}

\ifx 0
\paragraph{Case 1:} We look at when the first bullet from the lemma statement is satisfied. We have
\begin{equation}
\begin{aligned}
& g(\xx') - g(\xx)\\
& \leq \left\langle \nabla g(\xx), \oxx_S - \xx\right\rangle
+ \beta \left\|\oxx_S - \xx\right\|_{2}^2\\
& = \left\langle \nabla g(\xx), - \eta \nabla_S g(\xx)\right\rangle
+ \beta \left\|\eta \nabla_S g(\xx)\right\|_{2}^2\\
& = -(\eta - \beta\eta^2)\left\|\nabla_S g(\xx)\right\|_2^2\\
& = -(4\beta)^{-1}\left\|\nabla_S g(\xx)\right\|_2^2\,,
\end{aligned}
\label{eq:smoothness_ineq}
\end{equation}
where the last equation follows by our setting of $\eta = 1/(2\beta)$.
Let us from now on
assume that the first bullet of the statement is not satisfied, so
\begin{align*}
g(\oxx_S) - g(\xx) > - (8\kappa)^{-1} \left(g(\xx) - f(\xx^*)\right)\,.
\end{align*}
Combined with (\ref{eq:smoothness_ineq}), this implies:
\begin{align}
\left\|\nabla_S g(\xx)\right\|_2^2 \leq (\alpha/2) (g(\xx) - f(\xx^*))\,.
\label{eq:gradient_small}
\end{align}

\paragraph{Second bullet} Similarly, we examine what happens when the second bullet fails to hold.
\fi

We let $S$ be the support of $\xx$ and $S'$ the support of $\xx'$, i.e. the set of $s'$ indices of
the largest magnitude entries of the vector. Since $\nabla g(\xx)=\nabla f(\xx)+\beta \ww \xx$, we have
\begin{align*}
\oxx = 
\xx - \eta \nabla g(\xx)
= (\onev-0.5 \ww)\xx - \eta \nabla f(\xx)\,,
\end{align*}
where $\eta = (2\beta)^{-1}$.
We let $A = S'\backslash S$ be the newly inserted entries
and $B = S\backslash S'$ be the entries that were just removed from the support.
Note that 
\begin{align*}
\xx' 
& =  \left[\xx - \eta \nabla g(\xx)\right]_{S'}\\
& = \xx - \eta \nabla_{S'} g(\xx) - \xx_B \\
& = \xx - \eta \nabla_{S'\cup B} g(\xx) - \oxx_B \,.
\end{align*}
Using (\ref{eq:g_smoothness}), we have 
\begin{equation}
\begin{aligned}
& g(\xx') - g(\xx)\\
&  \leq \left\langle \nabla g(\xx), -\eta \nabla_{S'\cup B} g(\xx) - \oxx_B\right\rangle
{\ifdefined\arxiv
\else
\\ &
\fi}
+ \beta \left\|-\eta\nabla_{S'\cup B} g(\xx) - \oxx_B\right\|_{2}^2\\ 
& = 
- (4\beta)^{-1} \left\|\nabla_{S'\cup B} g(\xx)\right\|_2^2
+ \beta \left\|\oxx_B\right\|_{2}^2\\
& = 
- \beta \left\|\eta\nabla_{S\cup A} g(\xx)\right\|_2^2
+ \beta \left\|\oxx_B\right\|_{2}^2\\
& \leq 
- \beta \left\|\eta\nabla_{S\cup A'} g(\xx)\right\|_2^2
+ \beta \left\|\oxx_{B'}\right\|_{2}^2\,,
\label{eq:iht_smoothness}
\end{aligned}
\end{equation}
for any two sets $A'\in[n]\backslash S$ and $B'\subseteq S$ with $|A'| = |B'|$.
The latter inequality follows because of the following lemma about IHT:
\begin{lemma}
Suppose that we run one step of IHT on vector $\xx$ supported on $S$ for some function $g$,
and let
the updated solution vector be $\xx' = \oxx_{(S\cup A)\backslash B}$, where
$\oxx = \xx - \eta\nabla g(\xx)$.
Then, for any $A'\subseteq [n]\backslash S$ and $B'\subseteq S$ with $|A'|=|B'|$, we have
\begin{align}
- \left\|\eta\nabla_{A} g(\xx)\right\|_2^2
+ \left\|\oxx_B\right\|_{2}^2
 \leq 
- \left\|\eta\nabla_{A'} g(\xx)\right\|_2^2
+ \left\|\oxx_{B'}\right\|_{2}^2\,.
\label{eq:iht_optimality}
\end{align}
\end{lemma}
\begin{proof}
If we denote $|A| = |B| = t$ and $|A'|=|B'|=t'$, then note that by definition of IHT, $A$ are the $t$ largest entries in
\begin{align*}
\left|\oxx_{[n]\backslash S}\right|
= \eta \left|\nabla_{[n]\backslash S} g(\xx)\right| \,,
\end{align*}
and $B$ are the $t$ smallest entries in $\left|\oxx_{S}\right|$. Similarly, we can assume that $A'$ are the $t'$ largest entries
in $\eta \left|\nabla_{[n]\backslash S} g(\xx)\right|$
and $B'$ are the $t'$ smallest entries in $\left|\oxx_S\right|$, since this way the right hand side of 
(\ref{eq:iht_optimality}) takes its minimum value. If $t'=t$, we are done. We consider two cases:
\begin{enumerate}
\item{$t' > t$: In this case we have $A'\supseteq A$, $B'\supseteq B$, so
\begin{align*}
& 
- \left\|\eta\nabla_{A'} g(\xx)\right\|_2^2 + \left\|\oxx_{B'}\right\|_{2}^2
 + \left\|\eta\nabla_{A} g(\xx)\right\|_2^2 - \left\|\oxx_{B}\right\|_{2}^2\\
& = - \left\|\eta\nabla_{A'\backslash A} g(\xx)\right\|_2^2 + \left\|\oxx_{B'\backslash B}\right\|_{2}^2\\
& = - \left\|\oxx_{A'\backslash A}\right\|_2^2 + \left\|\oxx_{B'\backslash B}\right\|_{2}^2\\
& \geq (t'-t)\left(-\max_{i\in A'\backslash A} (\ox_i)^2 + \min_{j\in B'\backslash B} (\ox_j)^2\right)\\
&\geq 0 \,,
\end{align*}
where the last inequality follows since, by definition of the IHT step, 
$\left|\ox_i\right| \leq \left|\ox_j\right|$ for any $i\in A'\backslash A$ and $j\in B'\backslash B$. Otherwise,  $i$ would have taken $j$'s place in $S'$.
}
\item{$t' < t$: In this case we have $A' \subseteq A$, $B'\subseteq B$. Similarly to the previous case,
\begin{align*}
& 
- \left\|\eta\nabla_{A'} g(\xx)\right\|_2^2 + \left\|\oxx_{B'}\right\|_{2}^2
 + \left\|\eta\nabla_{A} g(\xx)\right\|_2^2 - \left\|\oxx_{B}\right\|_{2}^2\\
& = \left\|\eta\nabla_{A\backslash A'} g(\xx)\right\|_2^2 - \left\|\oxx_{B\backslash B'}\right\|_{2}^2\\
& = \left\|\oxx_{A\backslash A'}\right\|_2^2 - \left\|\oxx_{B\backslash B'}\right\|_{2}^2\\
& \geq (t-t')\left(\min_{i\in A\backslash A'} (\ox_i)^2 - \max_{j\in B\backslash B'} (\ox_j)^2\right)\\
&\geq 0 \,,
\end{align*}
where the last inequality follows since, by definition of the IHT step,
$\left|\ox_i\right| \geq \left|\ox_j\right|$ for any $i\in A\backslash A'$ and $j\in B\backslash B'$. Otherwise $i$ wouldn't have taken $j$'s place in $S'$.
}
\end{enumerate}
\end{proof}
Now, let us assume that the first bullet in the lemma statement is false, i.e.
\begin{align*}
g(\xx') - g(\xx) > -(16\kappa)^{-1} (g(\xx) - f(\xx^*))\,.
\end{align*}
Setting $A' = B' = \emptyset$ in (\ref{eq:iht_smoothness}), we get that
\begin{align*}
& g(\xx') - g(\xx) \leq 
- \beta \left\|\eta\nabla_{S} g(\xx)\right\|_2^2\,,
\end{align*}
so we conclude that
\begin{align}
\left\|\nabla_{S} g(\xx)\right\|_2^2< \frac{1}{16 \kappa \beta \eta^2} \left(g(\xx) - f(\xx^*)\right)
= \frac{\alpha}{4} \left(g(\xx) - f(\xx^*)\right)\,.
\label{eq:gradient_small}
\end{align}
Now, we again use (\ref{eq:iht_smoothness}) but we set 
$A'$ to be the $s$ entries from $[n]\backslash S$ on which
$\nabla g(\xx)$ has the largest magnitude,
and $B'$ to be the $s$ entries from $S$ on which $\oxx$ has the smallest magnitude. Also,
let $R$ be an arbitrary subset of $S\backslash S^*$ with size $r \geq 2s$.
We then have
\begin{equation}
\begin{aligned}
& g(\xx') - g(\xx)\\
& \leq - (4\beta)^{-1} \left\|\nabla_{S\cup A'} g(\xx)\right\|_2^2
+ \beta \left\|\oxx_{B'}\right\|_{2}^2\\
& \leq 
- (4\beta)^{-1} \left\|\nabla_{S\cup S^*} g(\xx)\right\|_2^2
+ \beta \left\|\oxx_{B'}\right\|_{2}^2\\
& \leq 
- (4\beta)^{-1} \left\|\nabla_{S\cup S^*} g(\xx)\right\|_2^2
+ \frac{\beta s}{|R\backslash S^*|} \left\|\oxx_{R\backslash S^*}\right\|_{2}^2\\
& \leq 
- (4\beta)^{-1} \left\|\nabla_{S\cup S^*} g(\xx)\right\|_2^2
+ \frac{\beta s}{r-s} \left\|\oxx_{R\backslash S^*}\right\|_{2}^2
\,,
\end{aligned}
\label{eq:smoothness}
\end{equation}
where we used the fact that 
\begin{align*}
\left\|\nabla_{A'} g(\xx)\right\|_2^2 \geq \left\|\nabla_{S^*\backslash S} g(\xx)\right\|_2^2
\end{align*}
by definition of $A'$ (and since $|S^*\backslash S|\leq s$), 
and the fact that, by definition of $B'$ (and since $|R\backslash S^*| \geq 2s - s = |B'|$),
\begin{align*}
 \frac{1}{|B'|} \left\|\oxx_{B'}\right\|_{2}^2
& \leq \frac{1}{|R\backslash S^*|} \left\|\oxx_{R\backslash S^*}\right\|_{2}^2\,.
\end{align*}
In fact, we will let 
$R = \{i\in S\ |\ w_i > 0\}$ be the set of elements that are being regularized.
To lower bound the size $r$ of this set, note that by the guarantee of the 
lemma statement,
\begin{align*}
n - s'/2
& \leq 
\left\|\ww\right\|_1 \\
& \leq n - \left|\{\text{$i\in S\ |\ w_i = 0$}\}\right|\\
& = n - (s' - r)\,,
\end{align*}
so $r \geq s'/2$. We conclude that   $r\geq 2s$ since $s' \geq 4s$.

Now, because of the fact that $f$ is $\alpha$-strongly convex, we have
\begin{equation}
\begin{aligned}
& f(\xx^*) - f(\xx) \\
& \geq \langle \nabla f(\xx), \xx^* - \xx\rangle + (\alpha/2) \left\|\xx^* - \xx\right\|_2^2\\
& = \langle \nabla g(\xx), \xx^* - \xx\rangle - \beta \langle \ww\xx, \xx^* - \xx\rangle
+ (\alpha/2) \left\|\xx^* - \xx\right\|_2^2\\
& \geq -\alpha^{-1} \left\|\nabla_{S\cup S^*} g(\xx)\right\|_2^2
{\ifdefined\arxiv
\else
\\ &
\fi}
- \beta \langle \ww\xx, \xx^* - \xx\rangle
+ (\alpha/4) \left\|\xx^* - \xx\right\|_2^2
\,,
\end{aligned}
\label{eq:strconvex}
\end{equation}
where we used the inequality
\[ \langle \aa, \bb\rangle + (\alpha/4)\left\|\bb\right\|_2^2 
\geq - \alpha^{-1}\left\|\aa\right\|_2^2\,. \]
By re-arranging and plugging (\ref{eq:strconvex}) into (\ref{eq:smoothness}), we get
\begin{equation}
\begin{aligned}
& g(\xx') - g(\xx)\\
& \leq
-(4\kappa)^{-1}
\Big(f(\xx) - f(\xx^*) 
{\ifdefined\arxiv
\else
\\ &
\fi}
- \beta \langle \ww\xx, \xx^*-\xx\rangle
 + (\alpha/4) \left\|\xx^* - \xx\right\|_2^2
\Big) 
{\ifdefined\arxiv
\else
\\ &
\fi}
+ \frac{\beta s}{r-s} \left\|\oxx_{R\backslash S^*}\right\|_{2}^2\\
& =
-(4\kappa)^{-1} \Big(g(\xx) - f(\xx^*) 
{\ifdefined\arxiv
\else
\\ &
\fi}
- \beta \langle \ww\xx, \xx^*-\xx\rangle - (\beta/2) \left\|\xx\right\|_{\ww,2}^2 
{\ifdefined\arxiv
\else
\\ &
\fi}
+ (\alpha/4) \left\|\xx^* - \xx\right\|_2^2
 - \frac{4\kappa\beta s}{r-s} \left\|\oxx_{R\backslash S^*}\right\|_2^2\Big)\,.
\end{aligned}
\label{eq:progress}
\end{equation}
Now, note that by definition of $\oxx$ we have
\begin{align*}
& \left\|\oxx_{R\backslash S^*}\right\|_2^2\\
& \leq 
2 \left\|\xx_{R\backslash S^*}\right\|_2^2
 + 2(2\beta)^{-2} \left\|\nabla_{R\backslash S^*} g(\xx)\right\|_2^2
\end{align*}
and, since $w_i\geq 1/2$ for each $i\in R$,
\begin{align*}
\left\|\xx_{R\backslash S^*}\right\|_2^2
\leq 2 \left\|\xx_{R\backslash S^*}\right\|_{\ww,2}^2
\leq 2 \left\|\xx_{S\backslash S^*}\right\|_{\ww,2}^2\,.
\end{align*}
Therefore,
 \begin{align*}
 & \frac{4\kappa\beta s}{r-s} \left\|\oxx_{R\backslash S^*}\right\|_2^2\\
 & \leq \frac{16\kappa\beta s}{r-s} \left\|\xx_{S\backslash S^*}\right\|_{\ww,2}^2
 + \frac{2\kappa s}{\beta (r-s)} \left\|\nabla_{R\backslash S^*} g(\xx)\right\|_2^2
 \\
 & \leq \frac{16\kappa\beta s}{r-s} \left\|\xx_{S\backslash S^*}\right\|_{\ww,2}^2
 + \frac{s}{2(r-s)} (g(\xx)-f(\xx^*))\,,
 \end{align*}
 where the last inequality follows from (\ref{eq:gradient_small})
 since $R\backslash S^* \subseteq S$.
 Plugging this back into
 (\ref{eq:progress}), we get
\begin{equation}
\begin{aligned}
& g(\xx') - g(\xx)\\
& \leq
-(4\kappa)^{-1} \Big(\left(1 - \frac{s}{2(r-s)}\right)\left(g(\xx) - f(\xx^*)\right) 
\\ &
- \beta \langle \ww\xx, \xx^*-\xx\rangle - (\beta/2) \left\|\xx\right\|_{\ww,2}^2 
{\ifdefined\arxiv
\else
\\ &
\fi}
+ (\alpha/4) \left\|\xx^* - \xx\right\|_2^2 - \frac{16\kappa\beta s}{r-s} \left\|\xx_{S\backslash S^*}\right\|_{\ww,2}^2\Big)\\
& =
-(4\kappa)^{-1} \Big(\left(1 - \frac{s}{2(r-s)}\right)\left(g(\xx) - f(\xx^*)\right) \\
& \underbrace{- \beta \langle \ww\xx_{S\cap S^*}, \xx^*-\xx\rangle
 + (\alpha/4) \left\|\xx^* - \xx\right\|_2^2}_{\geq -(\kappa\beta)\left\|\xx_{S\cap S^*}\right\|_{\ww^2,2}^2}
{\ifdefined\arxiv
\else
 \\ &
\fi}
 + \beta \left\|\xx_{S\backslash S^*}\right\|_{\ww,2}^2 - (\beta/2) \left\|\xx\right\|_{\ww,2}^2 
\underbrace{- \frac{16\kappa\beta s}{r-s} \left\|\xx_{S\backslash S^*}\right\|_{\ww,2}^2}_{
\geq -(\beta/4)\left\|\xx_{S\backslash S^*}\right\|_{\ww,2}^2}\Big)\\
& \leq
-(4\kappa)^{-1} \Big(\left(1 - \frac{s}{2(r-s)}\right)\left(g(\xx) - f(\xx^*)\right) 
{\ifdefined\arxiv
\else
\\ &
\fi}
- (\kappa\beta) \left\|\xx_{S\cap S^*}\right\|_{\ww^2,2}^2 
 - (\beta/2) \left\|\xx_{S\cap S^*}\right\|_{\ww,2}^2 
{\ifdefined\arxiv
\else
\\ &
\fi}
 + (\beta/4) \left\|\xx_{S\backslash S^*}\right\|_{\ww,2}^2\Big)\\
& \leq
-(4\kappa)^{-1} \Big(0.5\left(g(\xx) - f(\xx^*)\right) 
{\ifdefined\arxiv
\else
\\ &
\fi}
- (\kappa+1)\beta \left\|\xx_{S\cap S^*}\right\|_{\ww^2,2}^2 + (\beta/4) \left\|\xx_{S\backslash S^*}\right\|_{\ww,2}^2\Big)\,,
\end{aligned}
\end{equation}
where we used the fact that
\begin{align*}
\frac{s}{2(r-s)} \leq 1/2\,,
\end{align*}
which holds as long as  $s' \geq 4s$, and
\begin{align*}
\frac{16\kappa\beta s}{r-s} \leq \beta/4\,,
\end{align*}
which holds as long as $r\geq s'/2$ and $s' \geq (128\kappa + 2) s$.
In the last inequality we also used the property $\ww/2 \leq \ww^2$, which is by definition of $\ww$.

Now, note that, because we have assumed that the first bullet of the statement doesn't hold, it has
to be the case that
\begin{align*}
& (1/4)\left(g(\xx) - f(\xx^*)\right) 
{\ifdefined\arxiv
\else
\\ &
\fi}
- (\kappa+1)\beta \left\|\xx_{S\cap S^*}\right\|_{\ww^2,2}^2 + (\beta/4) \left\|\xx_{S\backslash S^*}\right\|_{\ww,2}^2 \leq 0\,.
\end{align*}
This immediately implies that
\begin{align*}
 &(\kappa+1)\beta \left\|\xx_{S\cap S^*}\right\|_{\ww^2,2}^2 \geq (\beta/4) \left\|\xx_{S\backslash S^*}\right\|_{\ww,2}^2\\
 &\Rightarrow
 (4\kappa + 4) \left\|\xx_{S\cap S^*}\right\|_{\ww^2,2}^2 \geq \left\|\xx_{S\backslash S^*}\right\|_{\ww,2}^2\\
 &\Rightarrow
 (4\kappa + 6) \left\|\xx_{S\cap S^*}\right\|_{\ww^2,2}^2 \geq \left\|\xx\right\|_{\ww,2}^2\,,
\end{align*}
so
\begin{align*}
 \left\|\xx_{S\cap S^*}\right\|_{\ww^2,2}^2 \geq (4\kappa+6)^{-1} \left\|\xx\right\|_{\ww,2}^2\,.
\end{align*}
Similarly we also have
\begin{align*}
& (\kappa+1)\beta \left\|\xx_{S\cap S^*}\right\|_{\ww^2,2}^2 
\geq (1/4)\left(g(\xx) - f(\xx^*)\right) \\
& \Rightarrow 
 (\beta/2) \left\|\xx_{S\cap S^*}\right\|_{\ww^2,2}^2 \geq
(8\kappa+8)^{-1} (g(\xx) - f(\xx^*))\,.
\end{align*}
Therefore the second bullet of the statement is true, and we are done.

\ifx 0

\[ r \geq s/2 \geq 16 (1-\delta)^{-1} (\beta/\alpha) s^*\]
and so (\ref{eq:progress1}) becomes
\begin{equation}
\begin{aligned}
& \geq (\beta/4) \left\|\xx_{S\backslash S^*}\right\|_{\ww,2}^2 
+ (16 \beta)^{-1} \left\|\nabla_{R\backslash S^*} g(\xx)\right\|_2^2
{\ifdefined\arxiv
\else
\\ &
\fi}
- \beta \langle \xx_{S\cap S^*}, \xx^*-\xx\rangle_{\ww}\,.
\end{aligned}
\label{eq:progress2}
\end{equation}
Now, we write
\begin{equation}
\begin{aligned}
& - \beta \langle \xx, \xx^*-\xx\rangle_{\ww} - (\beta/2) \left\|\xx\right\|_{\ww,2}^2 
{\ifdefined\arxiv
\else
\\ & 
\fi}
- (\beta^2 s^*/(\alpha (1-\delta) r)) \left\|\oxx_{R\backslash S^*}\right\|_{2}^2\\
& = (\beta/2) \left\|\xx_{S\backslash S^*}\right\|_{\ww,2}^2  
- (\beta^2 s^*/(\alpha (1-\delta) r)) \left\|\oxx_{R\backslash S^*}\right\|_{2}^2
{\ifdefined\arxiv
\else
\\ &
\fi}
- \beta \langle \xx_{S\cap S^*}, \xx^*-\xx\rangle_{\ww}
\,.
\end{aligned}
\label{eq:progress1}
\end{equation}
The last term together with the last term of (\ref{eq:progress}) can be bounded as
\begin{align*}
& - \beta \langle \xx_{S\cap S^*}, \xx^*-\xx\rangle_{\ww}
+ (\alpha\delta/2) \left\|\xx^* - \xx\right\|_2^2\\
& \geq - (\beta^2 / (2\alpha\delta))\left\|\xx_{S\cap S^*}\right\|_{\ww^2,2}^2
\end{align*}
Additionally, note that 
\[ (2\beta)^{-1}\left\|\nabla_{R\backslash S^*} g(\xx)\right\|_2^2
\leq (4\beta/\alpha)^{-1} (g(\xx) - f(\xx^*))\,. \]
Putting everything together, we have
\begin{equation}
\begin{aligned}
& g(\oxx_{(S\cup A)\backslash B}) - g(\xx)\\
& \leq
-(1-\delta)(2\beta/\alpha)^{-1} \Big((1/2)(g(\xx) - f(\xx^*)) 
{\ifdefined\arxiv
\else
\\ &
\fi}
+ (\beta/4) \left\|\xx_{S\backslash S^*}\right\|_{\ww,2}^2
- (\beta^2/(2\alpha\delta)) \left\|\xx_{S^*}\right\|_{\ww^2,2}^2
\Big)\,.
\end{aligned}
\label{eq:progress}
\end{equation}
Now, if 
\begin{align*}
(\beta/4) \left\|\xx_{S\backslash S^*}\right\|_{\ww,2}^2
- (\beta^2/(2\alpha\delta)) \left\|\xx_{S^*}\right\|_{\ww^2,2}^2 \geq 0\,,
\end{align*}
and setting $\delta=1/4$,
we get that 
\begin{align*}
& g(\oxx_{(S\cup A)\backslash B}) - g(\xx)\\
& \leq
-(16\beta/(3\alpha))^{-1} \left(g(\xx) - f(\xx^*)\right)\,,
\end{align*}
which satisfies the second bullet of the lemma statement.
Additionally, if 
\[ (\beta^2/(2\alpha\delta))\left\|\xx_{S^*}\right\|_{\ww^2,2}^2 \leq (1/4)(g(\xx) - f(\xx^*))\,, \]
then
\begin{align*}
& g(\oxx_{(S\cup A)\backslash B}) - g(\xx)\\
& \leq
-(32\beta/(3\alpha))^{-1} \left(g(\xx) - f(\xx^*)\right)\,,
\end{align*}
which also satisfies the second bullet.
\paragraph{Third bullet} By the above discussion, if neither of the first two bullets
are satisfied, we must have
\begin{align*}
\left\|\xx_{S^*}\right\|_{\ww^2,2}^2
\geq 
(8\beta/\alpha)^{-1} \left\|\xx_{S\backslash S^*}\right\|_{\ww,2}^2
\end{align*}
and
\begin{align*}
(\beta/2)\left\|\xx_{S^*}\right\|_{\ww^2,2}^2
\geq (16\beta/\alpha)^{-1} \left(g(\xx) - f(\xx^*)\right)\,.
\end{align*}
In the former, we can use the fact that $\ww^2 \geq \ww / 2$ to obtain
\begin{align*}
\left\|\xx_{S^*}\right\|_{\ww^2,2}^2
\geq 
(8\beta/\alpha + 2)^{-1} \left\|\xx\right\|_{\ww,2}^2\,.
\end{align*}
So the conditions of the third bullet are satisfied. 
\fi
\end{proof}

%% file: rank_final.tex
\section{Low Rank Minimization}
\label{sec:proofs_lowrank}

\subsection{Preliminaries}
We will use the following simple lemma about Frobenius products between low-rank projections and symmetric PSD matrices.
We remind the reader that $H_r\left(\AA\right)$ is the matrix consisting of the top $r$ components from the singular value decomposition of $\AA$.
\begin{lemma}
For any two symmetric PSD matrices $\vPi, \AA\in\mathbb{R}^{n\times n}$, where 
$\rank(\vPi) \leq r$ and $\left\|\vPi\right\|_2 \leq 1$, we have that
\begin{align*}
\left|\langle \vPi, \AA\rangle\right|
\leq \mathrm{Tr}\left[H_r(\AA)\right]\,.
\end{align*}
\label{fact:matrix_ineq}
\end{lemma}
\begin{proof}
We will use the following inequality for singular values
$$\sum_{i=1}^k\sigma_i(AB)\le \sum_{i=1}^k\sigma_i(A)\sigma_i(B)$$
for $k=1,\dots,n$, $A,B\in \mathbb{R}^{n\times n} $ and $\sigma_1(A)\ge \dots \ge \sigma_n(A)$ are singular values of matrix $A$ (see page 177 in \cite{H}).
Then
\begin{align*}
\left|\langle \vPi, \AA\rangle\right|&= \mathrm{Tr}\left[ \vPi \AA^\top\right]= \mathrm{Tr}\left[ \vPi \AA\right]=
\sum_{i=1}^n\sigma_i(\vPi \AA) \\
& \le \sum_{i=1}^n\sigma_i(\vPi)\sigma_i( \AA)
= \sum_{i=1}^r\sigma_i(\vPi)\sigma_i( \AA)\le  \sum_{i=1}^r \sigma_i( \AA)
=\mathrm{Tr}\left[H_r(\AA)\right].
\end{align*}
\ifx 
Let $\AA = \UU\vLambda \UU^\top$ be the eigendecomposition of $\AA$, and
let $\zz$ be the vectorized diagonal of $\UU^\top \vPi \UU$ (note that this is not necessarily a diagonal matrix)
and $\vlambda$ the vectorized diagonal of $\vLambda$. We have
\begin{align*}
& \left|\langle \vPi, \AA\rangle\right|
= \left|\langle \vPi, \UU \vLambda \UU^\top\rangle\right|
= \left|\langle \UU^\top \vPi \UU, \vLambda\rangle\right|
= \left|\langle \zz, \vlambda\rangle\right|\,,
\end{align*}
where the last equality follows by the fact that $\vLambda$ is diagonal, and so the Frobenius inner product only acts over the diagonal terms of $\UU^\top \vPi \UU$.
Now, it is a well known fact that $\OO \preceq \vPi \preceq \II$ implies
\begin{align}
\OO \preceq \UU^\top \vPi \UU \preceq \UU^\top \UU \preceq \II \,, 
\label{eq:proj_spectral_ineq}
\end{align}
because $\UU$ has orthonormal columns. As $\zz$ is the diagonal of $\UU^\top \vPi \UU$, 
(\ref{eq:proj_spectral_ineq}) implies that
for all $i$ we have
\[ z_i = \onev_i^\top \UU^\top \vPi\UU\onev_i \in [0,1]\,. \]
Additionally, 
\[ \left\|\zz\right\|_1 
= \mathrm{Tr}\left[\UU^\top \vPi \UU\right]
\leq  \mathrm{Tr}\left[\vPi\right] \leq \rank(\vPi) \left\|\vPi\right\|_2\leq r \,,\]
where we used the fact that 
\[ 
\mathrm{Tr}\left[\UU^\top \vPi \UU\right]
= \mathrm{Tr}\left[\vPi \UU\UU^\top \vPi\right]
\leq \mathrm{Tr}\left[\vPi^2 \right]
=  \mathrm{Tr}\left[\vPi\right] \,.\]
Therefore, we conclude that
\begin{align*}
 \left|\langle \zz, \vlambda\rangle\right|
& \leq \underset{\zz:\substack{\zerov\leq \zz\leq \onev \\ \left\|\zz\right\|_1\leq r}}{\max}\, \langle \zz, \vlambda\rangle\\
& = \left\|H_r(\vlambda)\right\|_1\\
& = \mathrm{Tr}\left[H_r(\AA)\right]\,,
\end{align*}
where $H_r(\vlambda)$ is a thresholding operator that keeps the top $r$ entries (in absolute value) of $\vlambda$, and we used the fact that $\vlambda\geq \zerov$ because $\AA$ is PSD
and $\vlambda$ are its eigenvalues.
\fi
\end{proof}

\subsection{Analysis}

This section is devoted to proving Theorem~\ref{thm:rank_ARHT_plus}, which analyzes 
an algorithm for low rank optimization that uses adaptive regularization.
\begin{reptheorem}{thm:rank_ARHT_plus}[Adaptive Regularization for Low Rank Optimization]
Let $f\in\mathbb{R}^{m\times n}\rightarrow\mathbb{R}$ be a
convex function with condition number $\kappa$ and
consider the low rank minimization
problem
\begin{align}
\underset{\rank(\AA) \leq r}{\min}\, f(\AA)\,.
\label{eq:rank_objective_intro}
\end{align}
For any error parameter $\eps>0$,
there exists a polynomial time algorithm that returns a matrix $\AA$ with
$\rank(\AA) \leq O\left(r\left(\kappa+\log \frac{f(\OO)-f(\AA^*)}{\eps}\right)\right)$
and 
$f(\AA) \leq f(\AA^*) + \eps$, where $\AA^*$ is any rank-$r$ matrix.
\end{reptheorem}
\begin{proof}[Proof of Theorem~\ref{thm:rank_ARHT_plus}]
Let the smoothness and strong convexity parameters of $f$ be $\beta,\alpha$.
We repeatedly apply Lemma~\ref{lem:rank_step}
$T\geq O\left(r\kappa \log \frac{f(\AA^0) + (\beta/2)\left\|\AA^0\right\|_F^2 - f(\AA^*)}{\epsilon}\right)$
times starting from solution $\AA^0=\OO$
and weight matrices $\WW^0 = \II$, $\YY^0=\II$. Thus, we obtain
solutions $\AA^0,\dots,\AA^{T}$, and weights
$\WW^0,\WW^1,\dots,\WW^{T}$ and $\YY^0,\YY^1,\dots,\YY^T$.
We let 
\begin{align*}
& g^t(\AA) \\
& = f(\AA) + (\beta/4) \left(\langle\WW^t, \AA^t(\AA^t)^\top \rangle + 
\langle \YY^t, (\AA^t)^\top \AA^\top\rangle\right)
\end{align*} be the regularized function at iteration $t$.

We denote by $T_i$ the total number of iterations for which item $i\in\{1,2,3\}$ from the statement of Lemma~\ref{lem:rank_step} holds.

Consider the $T_2$ iterations for which 
item $2$ from the statement of Lemma~\ref{lem:rank_step} holds.
Without loss of generality, $\WW$ is updated at least $T_2/2$ times.
Letting $\AA^* = \UU^* \vSigma^* \VV^{*\top}$ be the singular value decomposition of $\AA^*$,
for each such iteration
we have 
\begin{align*}
& \mathrm{Tr}\left[\vPi_{\im(\UU^*)}\WW^{t+1}\vPi_{\im(\UU^*)}\right]\\
& \leq \mathrm{Tr}\left[\vPi_{\im(\UU^*)}\WW^{t}\vPi_{\im(\UU^*)}\right]
- (10\kappa)^{-1}\,,
\end{align*}
and for all other types of iterations we have
$\WW^{t+1}\preceq \WW^t$.
Therefore,
\begin{align*}
& \mathrm{Tr}\left[\vPi_{\im(\UU^*)}\WW^{T}\vPi_{\im(\UU^*)}\right]\\
& \leq \mathrm{Tr}\left[\vPi_{\im(\UU^*)}\WW^{0}\vPi_{\im(\UU^*)}\right]
- \frac{T_2}{2}(10\kappa)^{-1}\,.
\end{align*}
However, note that by the guarantee of Lemma~\ref{lem:rank_step} that $\WW^{T} \succeq \OO$, we have
\[ \mathrm{Tr}\left[\vPi_{\im(\UU^*)}\WW^{T}\vPi_{\im(\UU^*)}\right] \geq 0 \,,\]
and because $\WW^0 = \II$ we also know that 
\[ 
\mathrm{Tr}\left[\vPi_{\im(\UU^*)}\WW^{0}\vPi_{\im(\UU^*)}\right] 
=
\mathrm{Tr}\left[\vPi_{\im(\UU^*)}\right] 
\leq r
\,.\]
This implies that $T_2 \leq 20\kappa r$.

Now, if $T_1 \geq 16r\kappa\log \frac{g^0(\AA^0) - f(\AA^*)}{\epsilon}$, and since $g^t(\AA^t)$ is 
non-increasing for all $t$, we have
\begin{align*}
& g^T(\AA^T) - f(\AA^*)\\
& \leq \left(1 - (16r\kappa)^{-1}\right)^{T_1} (g^0(\AA^0) - f(\AA^*))\\
& \leq \epsilon\,,
\end{align*}
so $T_1 \leq 16r\kappa\log \frac{g^0(\AA^0) - f(\AA^*)}{\epsilon}$.

Similarly, if $T_3 \geq 10r\log \frac{g^0(\AA^0) - f(\AA^*)}{\epsilon}$ we have
\begin{align*}
& g^T(\AA^T) - f(\AA^*)\\
& \leq \left(1 - (10r)^{-1}\right)^{T_4} (g^0(\AA^0) - f(\AA^*))\\
& \leq \epsilon\,,
\end{align*}
so $T_3 \geq 10r\log \frac{g^0(\AA^0) - f(\AA^*)}{\epsilon}$.

Overall, we have that the total number of iterations is
\[ T = \sum T_i \leq 36r(\kappa+1) \log \frac{g^0(\AA^0) - f(\AA^*)}{\epsilon} \,. \]

The only thing left is to ensure that the conditions
\begin{align*}
& \mathrm{Tr}\left[\II - \WW^t\right] \leq r'/2\\
& \mathrm{Tr}\left[\II - \YY^t\right] \leq r'/2
\end{align*}
of Lemma~\ref{lem:rank_step} are satisfied for all $t$. By the guarantees of Lemma~\ref{lem:rank_step}, if one of items $2,3$ holds, then 
\begin{align*}
\mathrm{Tr}\left[\II - \WW^{t+1}\right] \leq 
\mathrm{Tr}\left[\II - \WW^{t}\right] + 1\,,
\end{align*}
and if item $1$ holds, then 
\[ \mathrm{Tr}\left[\II - \WW^{t+1}\right] =
\mathrm{Tr}\left[\II - \WW^{t}\right] \,. \]
As $\mathrm{Tr}\left[\II - \WW^0\right] = 0$, we have
\begin{align*}
& \mathrm{Tr}\left[\II - \WW^{T}\right] \\
& \leq T_2 + T_3 \\
& \leq 20\kappa r + 10 r \log \frac{g^0(\AA^0) - f(\AA^*)}{\epsilon}\\
& \leq r'/2\,,
\end{align*}
where the last inequality holds as long as
\begin{align*}
r' \geq 20r\left(2\kappa + \log \frac{g^0(\AA^0) - f(\AA^*)}{\epsilon}\right)\,.
\end{align*}
\end{proof}

\begin{lemma}[Low rank minimization step analysis]
Let $f:\mathbb{R}^{m\times n}\rightarrow \mathbb{R}$ be a $\beta$-smooth and $\alpha$-strongly convex function with condition
number $\kappa=\beta/\alpha$, 
and $\WW\in\mathbb{R}^{m\times m}, \YY\in\mathbb{R}^{n\times n}$ be symmetric positive semi-definite weight matrices
with spectral norm bounded by $1$ and such that
$\mathrm{Tr}\left[\II-\WW\right] \leq r'/2$ and $ 
\mathrm{Tr}\left[\II-\YY\right] \leq r'/2$ for fixed parameter $r'\geq 256 r$.
We define the regularized function 
\begin{align*}
g(\AA) := f(\AA) + \underbrace{(\beta/4) \left(\langle \WW, \AA\AA^\top\rangle + \langle \YY, \AA^\top \AA\rangle \right)}_{\Phi(\AA)}\,.
\end{align*}

Now, consider a rank-$r'$ matrix $\AA\in\mathbb{R}^{m\times n}$
with singular value decomposition
\begin{align*}
\AA = \UU\vLambda \VV^\top = \sum\limits_{j\in S} \lambda_j \uu_j \vv_j^\top
\end{align*}
and with the property that 
\begin{align*}
\vPi_{\im(\UU)} \cdot \nabla g(\AA) \cdot \vPi_{\im(\VV)} = \OO\,.
\end{align*}

For any rank-$r$ solution $\AA^*$
where $r' \geq 256 r$,
there is a procedure that updates $\AA,\WW,\YY$,
and for which exactly one of the following scenarios holds:
\begin{enumerate}
\item{$\AA$ is updated to a rank-$r'$ matrix $\AA'$, and $\WW,\YY$ are not updated.
We have sufficient progress in the regularized function:
\begin{align*}
g(\AA') \leq g(\AA) - (16\kappa r)^{-1} \left(g(\AA) - f(\AA^*)\right) \,.
\end{align*}
}
\item{Exactly one of $\WW$ or $\YY$ is updated (wlog $\WW$) to a symmetric PSD
$\WW'\preceq \WW$, and $\AA$ is not updated. We have
\begin{align*}
& \mathrm{Tr}\left[\II-\WW'\right] \leq \mathrm{Tr}[\II-\WW] + 1
\end{align*}
and
\begin{align*}
& \mathrm{Tr}\left[\vPi_{\im(\UU^*)}\WW'\vPi_{\im(\UU^*)}\right] \\
& \leq \mathrm{Tr}\left[\vPi_{\im(\UU^*)}\WW\vPi_{\im(\UU^*)}\right] - (10\kappa)^{-1}\,.
\end{align*}
Respectively, for $\YY$:
\begin{align*}
& \mathrm{Tr}\left[\vPi_{\im(\VV^*)}\YY'\vPi_{\im(\VV^*)}\right] \\
& \leq \mathrm{Tr}\left[\vPi_{\im(\VV^*)}\YY\vPi_{\im(\VV^*)}\right] - (10\kappa)^{-1}\,.
\end{align*}
}
\item{Exactly one of $\WW$ or $\YY$ is updated (wlog $\WW$) to a symmetric PSD
$\WW'\preceq \WW$, and $\AA$ is not updated.
We have sufficient progress in the regularized function, where $g'$ is the regularized function with the new weights:
\begin{align*}
g'(\AA) \leq g(\AA) - (10 r)^{-1} \left(g(\AA) - f(\AA^*)\right) \,.
\end{align*}
Additionally,
\begin{align*}
& \mathrm{Tr}\left[\II-\WW'\right] \leq \mathrm{Tr}[\II-\WW] + 1
\end{align*}
}
\end{enumerate}
\label{lem:rank_step}
\end{lemma}
\begin{proof}
We attempt to make the update $\AA\rightarrow \AA'$ as defined in Lemma~\ref{lem:rank_step_cases}. If it makes enough progress, i.e.
\begin{align*}
g(\AA') \leq g(\AA) - (16\kappa r)^{-1} \left(g(\AA) - f(\AA^*)\right) \,,
\end{align*}
we are done. Otherwise, one of the items $2$-$5$ in the statement of Lemma~\ref{lem:rank_step_cases} must hold. Let us take them one by one.
\paragraph{Item 2:
\[ \langle \vPi_{\im(\UU^*)}, \WW \AA \AA^\top \WW\rangle \geq \left(10\kappa\right)^{-1} \langle \WW, \AA\AA^\top\rangle \,. \]
}
We update $\WW$ as 
\begin{align*}
\WW' = \WW - c\cdot \WW \AA \AA^\top \WW\,,
\end{align*}
where $c = \langle \WW,\AA\AA^\top\rangle^{-1}$.
Note that this update preserves symmetry, and
\begin{align*}
\OO \preceq \WW' \preceq \WW\,.
\end{align*}
This is because 
\begin{align*}
& c\WW^{1/2} \AA\AA^\top \WW^{1/2} \preceq c \langle \WW, \AA\AA^\top\rangle\cdot\II \preceq \II\,,
\end{align*}
so
\begin{align*}
\WW' = \WW^{1/2}\left(\II - c\WW^{1/2} \AA\AA^\top \WW^{1/2}\right)\WW^{1/2} \succeq \OO
\end{align*}
and
\[ \WW' = \WW - c\WW\AA\AA^\top\WW \preceq \WW\,.\]

Now, note that
\begin{align*}
\mathrm{Tr}\left[\II-\WW'\right]
& = \mathrm{Tr}\left[\II-\WW\right] + c \langle \WW^2, \AA\AA^\top\rangle\\
& \leq \mathrm{Tr}\left[\II-\WW\right] + c \langle \WW, \AA\AA^\top\rangle\\
& = \mathrm{Tr}\left[\II-\WW\right] + 1\,,
\end{align*}
where we used the fact that $\WW^2 \preceq \WW$, and (letting 
$\vPi^* = \vPi_{\im(\UU^*)}$ for convenience),
\begin{equation}
\begin{aligned}
\mathrm{Tr}\left[\vPi^* \WW' \vPi^*\right]
& = \mathrm{Tr}\left[\vPi^*\WW\vPi^*\right] - c \langle \vPi^*, \WW\AA\AA^\top\WW\rangle\\
& \leq \mathrm{Tr}\left[\vPi^* \WW \vPi^*\right] - c/(10\kappa) \langle \WW, \AA\AA^\top\rangle\\
& = \mathrm{Tr}\left[\vPi^* \WW \vPi^*\right] - (10\kappa)^{-1}\,,
\end{aligned}
\label{eq:rank_trace_decrease1}
\end{equation}

\paragraph{Item 3:
\[ \langle \vPi_{\im(\VV^*)}, \YY \AA^\top \AA \YY\rangle \geq \left(10\kappa\right)^{-1} \langle \YY, \AA^\top \AA\rangle \,. \]
}
This is entirely analogous to the previous case.

\paragraph{Item 4:
\begin{align}
(\beta/4)\, \mathrm{Tr}\left[H_r\left(\AA^\top \WW \AA\right)\right] 
\geq 10^{-1} \left(g(\AA) - f(\AA^*)\right) \,. 
\label{eq:rank_nonprogress_case1}
\end{align}
}

After considering the eigendecomposition 
\[ \WW^{1/2} \AA \AA^\top \WW^{1/2} = \sum\limits_{i\in[r']} \olam_i \ovv_i\ovv_i^\top \]
with $\olam_1 \geq \olam_2\geq \dots\geq \olam_{r'}\geq 0$,
(\ref{eq:rank_nonprogress_case1})
 can be re-phrased as
\begin{align*}
(\beta/4)\sum\limits_{i\in [r]} \olam_i > (1/10) \left(g(\AA) - f(\AA^*)\right)\,.
\end{align*}
We update $\WW$ as
\begin{align*}
\WW' = \WW^{1/2} \left(\II - r^{-1} \sum\limits_{i\in[r]} \ovv_i\ovv_i^\top\right)\WW^{1/2}
\end{align*}
and let $g'$ be the new regularized objective.
First of all, note that this operation preserves symmetry, 
and that $\OO \preceq \WW'\preceq \II$, since $\sum\limits_{i\in[r]} \ovv_i\ovv_i^\top \preceq \II$.
Additionally, 
\begin{align*}
\mathrm{Tr}\left[\II-\WW'\right]
& = \mathrm{Tr}\left[\II-\WW\right] + r^{-1} \sum\limits_{i\in [r]} \ovv_i^\top \WW \ovv_i\\
& \leq \mathrm{Tr}\left[\WW\right] + 1
\end{align*}
and
\begin{align*}
& g'(\AA) - g(\AA) \\
& = (\beta/4) \langle \WW', \AA\AA^\top\rangle - 
 (\beta/4)\langle \WW, \AA\AA^\top\rangle\\
& = - (\beta/(4r))\left\langle \WW^{1/2} \left(\sum\limits_{i\in [r]} \ovv_i\ovv_i^\top\right) \WW^{1/2}, \AA \AA^\top\right\rangle\\
& = - (\beta/(4r)) \sum\limits_{i\in[r]} \olam_i\\
& \leq -(10r)^{-1} (g(\AA) - f(\AA^*))\,,
\end{align*}

\paragraph{Item 5:
\[ (\beta/4) \, \mathrm{Tr}\left[H_r\left(\AA \YY \AA^\top\right)\right] \geq 10^{-1}\left(g(\AA) - f(\AA^*)\right) \,.\]
}
This is entirely analogous to the previous case.
\end{proof}

\begin{lemma}
Let $f:\mathbb{R}^{m\times n}\rightarrow \mathbb{R}$ be a $\beta$-smooth and $\alpha$-strongly convex function with condition
number $\kappa=\beta/\alpha$, 
and $\WW\in\mathbb{R}^{m\times m}, \YY\in\mathbb{R}^{n\times n}$ be symmetric positive semi-definite weight matrices
with spectral norm bounded by $1$ and such that
$\mathrm{Tr}\left[\II-\WW\right],
\mathrm{Tr}\left[\II-\YY\right] \leq r'/2$ for some parameter $r'\geq 0$.
We define the regularized function 
\begin{align*}
g(\AA) := f(\AA) + \underbrace{(\beta/4) \left(\langle \WW, \AA\AA^\top\rangle + \langle \YY, \AA^\top \AA\rangle \right)}_{\Phi(\AA)}\,.
\end{align*}

Now, consider a rank-$r'$ matrix $\AA\in\mathbb{R}^{m\times n}$
with singular value decomposition
\begin{align*}
\AA = \UU\vLambda \VV^\top = \sum\limits_{j\in S} \lambda_j \uu_j \vv_j^\top
\end{align*}
and with the property that 
\begin{align*}
\vPi_{\im(\UU)} \cdot \nabla g(\AA) \cdot \vPi_{\im(\VV)} = \OO\,.
\end{align*}
We define an updated solution
\begin{align*}
\AA' = \AA - \eta \cdot H_1(\nabla g(\AA)) - \lambda_j \uu_j \vv_j^\top\,,
\end{align*}
where $\eta=(2\beta)^{-1}$, $H_1(\cdot)$ returns the top singular component,
and $j\in S$ is picked to minimize $\lambda_j$. 

Then, for any rank-$r$ solution $\AA^*$,
where $r' \geq 256 r$,
and its singular value decomposition $\AA^* = \UU^*\vLambda^* \VV^{*\top}$,
at least one of the following conditions holds:
\begin{enumerate}
\item{We have sufficient progress in the regularized function:
\begin{align*}
g(\AA') \leq g(\AA)
 - (16\kappa r)^{-1} \left(g(\AA) - f(\AA^*)\right) \,.
\end{align*}
}
\item{$\WW \AA \AA^\top \WW$ is significantly correlated to $\UU^*$:
\[ \langle \vPi_{\im(\UU^*)}, \WW \AA \AA^\top \WW\rangle \geq \left(10\kappa\right)^{-1} \langle \WW, \AA\AA^\top\rangle \,. \]
}
\item{$\YY \AA^\top \AA \YY$ is significantly correlated to $\VV^*$:
\[ \langle \vPi_{\im(\VV^*)}, \YY \AA^\top \AA \YY\rangle \geq \left(10\kappa\right)^{-1} \langle \YY, \AA^\top\AA\rangle \,. \]
}
\item{The spectrum of $\AA^\top \WW \AA$ is highly concentrated and responsible for a constant fraction of the error:
\[ (\beta/4)\, \mathrm{Tr}\left[H_r\left(\AA^\top \WW \AA\right)\right] \geq 10^{-1} \left(g(\AA) - f(\AA^*)\right) \,. \]
and 
}
\item{The spectrum of $\AA \YY \AA^\top$ is highly concentrated and responsible for a constant fraction of the error:
\[ (\beta/4) \, \mathrm{Tr}\left[H_r\left(\AA \YY \AA^\top\right)\right] \geq 10^{-1}\left(g(\AA) - f(\AA^*)\right) \,.\]
}
\end{enumerate}
\label{lem:rank_step_cases}
\end{lemma}
\begin{proof}
Note that $g$ is a $2\beta$-smooth function. This follows because
\[ \nabla g(\AA) = \nabla f(\AA) + (\beta/2) \left(\WW \AA + \AA \YY\right)\,, \]
and so for any two matrices $\AA,\AA'$,
\begin{align*}
& \left\|\nabla g(\AA') - \nabla g(\AA)\right\|_F \\
& \leq 
\left\|\nabla f(\AA') - \nabla f(\AA)\right\|_F + (\beta/2) \left\|\WW(\AA'-\AA)\right\|_F
+(\beta/2)\left\|(\AA'-\AA)\YY\right\|_F\\
& \leq 
2\beta \left\|\AA'-\AA\right\|_F\,, 
\end{align*}
which is known to imply $2\beta$-smoothness of $g$.
Here we used the triangle inequality and the fact that $\WW,\YY\preceq \II$. Therefore, we have
\begin{equation}
\begin{aligned}
& g(\AA') - g(\AA) \\
& \leq \langle \nabla g(\AA), \AA' - \AA\rangle + 
\left\|\nabla g(\AA')-\nabla g(\AA)\right\|_F \left\|\AA'-\AA\right\|_F \\
& \leq \langle \nabla g(\AA), \AA' - \AA\rangle + \beta \left\|\AA' - \AA\right\|_F^2\\
& \leq -\eta \left\|\nabla g(\AA)\right\|_2^2
+ 2\beta \eta^2\left\|\nabla g(\AA)\right\|_2^2 + 2\beta \lambda_j^2\\
& = -(8\beta)^{-1} \left\|\nabla g(\AA)\right\|_2^2 + 2\beta \lambda_j^2\,,
\end{aligned}
\label{eq:rank_smooth}
\end{equation}
where in the second inequality we used the facts that
\begin{align*}
& \langle \nabla g(\AA), -\lambda_j \uu_j\vv_j^\top\rangle \\
& = \langle \vPi_{\im(\UU)}\nabla g(\AA)\vPi_{\im(\VV)}, -\lambda_j \uu_j\vv_j^\top\rangle \\
& =0
\end{align*}
and that, for any two matrices $\BB,\CC$, 
\[ \left\|\BB+\CC\right\|_F^2 \leq 2 \left\|\BB\right\|+ 2\left\|\CC\right\|_F^2\,. \]
The last equality follows by our choice of $\eta$.
In order to lower bound $\left\|\nabla g(\AA)\right\|_2^2$, we use the strong convexity of $f$ as follows:
\begin{equation}
\begin{aligned}
& f(\AA^*) - f(\AA) \\
& \geq \langle \nabla f(\AA), \AA^* - \AA\rangle + (\alpha/2) \left\|\AA^* - \AA\right\|_F^2\\
& = 
\langle \nabla g(\AA), \AA^* - \AA\rangle 
{\ifdefined\arxiv
\else 
\\ &
\fi}
- \langle\nabla \Phi(\AA), \AA^* - \AA \rangle
+ (\alpha/2) \left\|\AA^* - \AA\right\|_F^2\\
& = 
\underbrace{\langle \nabla g(\AA), \AA^* - \AA\rangle + (\alpha/4) \left\|\AA^* - \AA\right\|_F^2}_{P}
{\ifdefined\arxiv
\else
\\ &
\fi}
- \langle\nabla \Phi(\AA), \AA^* - \AA \rangle + (\alpha/4) \left\|\AA^* - \AA\right\|_F^2\,.
\end{aligned}
\label{eq:rank_strconv}
\end{equation}
\paragraph{Bounding $P$.}

We let $\vPi_{\im(\UU)}$, 
$\vPi_{\im(\VV)}$
be the orthogonal projections onto the images of $\UU$ and $\VV$ respectively, so we can write
\begin{align*}
& \AA^* - \AA\\
& = \vPi_{\im(\UU)} \left(\AA^* - \AA\right)\vPi_{\im(\VV)}
{\ifdefined\arxiv
\else
\\ &
\fi}
+ \left(\II - \vPi_{\im(\UU)}\right) \left(\AA^* - \AA\right)\vPi_{\im(\VV)}
{\ifdefined\arxiv
\else 
\\ &
\fi}
+ \left(\AA^* - \AA\right) \left(\II - \vPi_{\im(\VV)}\right)\\
& = \vPi_{\im(\UU)} \left(\AA^* - \AA\right)\vPi_{\im(\VV)}
{\ifdefined\arxiv
\else
\\ &
\fi}
+ \left(\II - \vPi_{\im(\UU)}\right) \AA^*\vPi_{\im(\VV)}
{\ifdefined\arxiv
\else
\\ &
\fi}
+ \AA^* \left(\II - \vPi_{\im(\VV)}\right)\,.
\end{align*}
Now, note that
\begin{align*}
& \langle \nabla g(\AA), \AA^* - \AA\rangle\\
& = \langle \nabla g(\AA), \left(\II - \vPi_{\im(\UU)}\right) \AA^*\vPi_{\im(\VV)}\rangle 
{\ifdefined\arxiv
\else
\\ &
\fi}
+ \langle \nabla g(\AA), \AA^* \left(\II - \vPi_{\im(\VV)}\right)\rangle\,,
\end{align*}
where we used the fact that 
\begin{align*}
& \langle\nabla g(\AA), \vPi_{\im(\UU)} (\AA^* - \AA) \vPi_{\im(\VV)}\rangle\\
& = \langle\vPi_{\im(\UU)} \nabla g(\AA) \vPi_{\im(\VV)}, \AA^* - \AA\rangle\\
& = 0\,,
\end{align*}
and 
\begin{align*}
& (\alpha/4) \left\|\AA^* - \AA\right\|_F^2\\
& \geq (\alpha/4) \left\|\left(\II - \vPi_{\im(\UU)}\right) \AA^*\vPi_{\im(\VV)}\right\|_F^2
{\ifdefined\arxiv
\else
\\ & 
\fi}
+ (\alpha/4) \left\|\AA^* \left(\II - \vPi_{\im(\VV)}\right)\right\|_F^2\,.
\end{align*}
Additionally, note that for any rank-$r$ matrix $\BB$, we have
\begin{align*}
& \langle \nabla g(\AA), \BB\rangle
+ (\alpha/4) \left\|\BB\right\|_F^2\\
&  \geq -\alpha^{-1} \left\|H_r\left(\nabla g(\AA)\right)\right\|_F^2\\
&  \geq -\alpha^{-1}r \left\|\nabla g(\AA)\right\|_2^2\,,
\end{align*}
a proof of which can be found e.g. in Lemma A.6 of~\cite{axiotis2021local}.
Applying this inequality with 
\[ \BB=\left(\II - \vPi_{\im(\UU)}\right) \AA^*\vPi_{\im(\VV)} \]
and 
\[ \BB=\AA^* \left(\II - \vPi_{\im(\VV)}\right)\] 
and summing them up,
we obtain
\begin{align*}
 P 
& = \langle \nabla g(\AA), \AA^* - \AA\rangle + (\alpha/4) \left\|\AA^* - \AA\right\|_F^2\\
& \geq -2\alpha^{-1}r \left\|\nabla g(\AA)\right\|_2^2\,.
\end{align*}
Plugging this into (\ref{eq:rank_strconv}) and re-arranging, we get
\begin{equation}
\begin{aligned}
& \left\|\nabla g(\AA)\right\|_2^2\\
& \geq \alpha / (2r) \Big( f(\AA) - f(\AA^*) 
{\ifdefined\arxiv
\else
\\ &
\fi}
- \langle\nabla\Phi(\AA), \AA^* - \AA \rangle
+ (\alpha/4) \left\|\AA^* - \AA\right\|_F^2\Big)\\
& = \alpha / (2r) \Big( g(\AA) - f(\AA^*) - \Phi(\AA)
{\ifdefined\arxiv
\else
\\ &
\fi}
\underbrace{- \langle\nabla\Phi(\AA), \AA^* - \AA \rangle
+ (\alpha/4) \left\|\AA^* - \AA\right\|_F^2}_{Q}\Big)\,.
\end{aligned}
\label{eq:rank_gradient}
\end{equation}
\paragraph{Bounding $Q$.}
We know that
\begin{align*}
- \langle\nabla\Phi(\AA), \AA^* - \AA \rangle
= - (\beta/2)\langle \WW \AA + \AA\YY, \AA^* - \AA \rangle\,.
\end{align*}
If we let
\begin{align*}
\AA^* = \UU^* \vLambda^* \VV^{*\top}
\end{align*}
be the SVD of $\AA^*$ and 
$\vPi_{\im(\UU^*)}$,
$\vPi_{\im(\VV^*)}$ be the orthogonal projections onto the images of $\UU^*$ and $\VV^*$ respectively, then we have
\begin{align*}
& - (\beta/2) \langle\WW \AA, \AA^* - \AA \rangle \\
& = - (\beta/2) \langle\WW \AA, \vPi_{\im(\UU^*)}(\AA^* - \AA)\vPi_{\im(\VV^*)}\rangle 
{\ifdefined\arxiv
\else
\\ &
\fi}
+ (\beta/2) \langle\WW, \AA\AA^\top \rangle 
 - (\beta/2) \langle\WW \AA, \vPi_{\im(\UU^*)}\AA\vPi_{\im(\VV^*)} \rangle\,. 
\end{align*}
Looking at the first term of this, we have
\begin{align*}
& - (\beta/2) \langle\WW \AA, \vPi_{\im(\UU^*)}(\AA^* - \AA)\vPi_{\im(\VV^*)}\rangle
{\ifdefined\arxiv
\else
\\ &
\fi}
+ (\alpha/8) \left\|\AA^* - \AA\right\|_F^2\\
& = - (\beta/2) \langle\vPi_{\im(\UU^*)}\WW \AA \vPi_{\im(\VV^*)}, \AA^* - \AA\rangle 
{\ifdefined\arxiv
\else
\\ &
\fi}
+ (\alpha/8) \left\|\AA^* - \AA\right\|_F^2\\
& \geq - \beta^2/(2\alpha) \left\|\vPi_{\im(\UU^*)}\WW \AA\vPi_{\im(\VV^*)}\right\|_F^2\,.
\end{align*}
Similarly for the terms containing $\YY$, we get
\begin{align*}
& - (\beta/2) \langle\AA\YY, \AA^* - \AA \rangle \\
& = - (\beta/2) \langle\AA\YY, \vPi_{\im(\UU^*)}(\AA^* - \AA)\vPi_{\im(\VV^*)}\rangle 
{\ifdefined\arxiv
\else
\\ &
\fi}
+ (\beta/2) \langle\YY, \AA^\top \AA\rangle 
 - (\beta/2) \langle\AA\YY, \vPi_{\im(\UU^*)}\AA\vPi_{\im(\VV^*)} \rangle\,. 
\end{align*}
and
\begin{align*}
& - (\beta/2) \langle\AA\YY, \vPi_{\im(\UU^*)}(\AA^* - \AA)\vPi_{\im(\VV^*)}\rangle
{\ifdefined\arxiv
\else
\\ &
\fi}
+ (\alpha/8) \left\|\AA^* - \AA\right\|_F^2\\
& \geq - \beta^2/(2\alpha) \left\|\vPi_{\im(\UU^*)}\AA\YY\vPi_{\im(\VV^*)}\right\|_F^2\,.
\end{align*}
In summary, we have
\begin{equation}
\begin{aligned}
 Q 
& = -\langle\nabla\Phi(\AA), \AA^* - \AA \rangle + (\alpha/4) \left\|\AA^* - \AA\right\|_F^2\\
& \geq - \beta^2/(2\alpha) \left\|\vPi_{\im(\UU^*)}\WW \AA\vPi_{\im(\VV^*)}\right\|_F^2
{\ifdefined\arxiv
\else
\\ &
\fi}
- \beta^2/(2\alpha) \left\|\vPi_{\im(\UU^*)}\AA\YY\vPi_{\im(\VV^*)}\right\|_F^2
\\ &
+ (\beta/2) \langle\WW, \AA\AA^\top \rangle 
 + (\beta/2) \langle\YY, \AA^\top \AA\rangle
\\ &
- (\beta/2) \langle\WW \AA, \vPi_{\im(\UU^*)}\AA\vPi_{\im(\VV^*)} \rangle
{\ifdefined\arxiv
\else
\\ &
\fi}
- (\beta/2) \langle\AA\YY, \vPi_{\im(\UU^*)}\AA\vPi_{\im(\VV^*)} \rangle\,. 
\end{aligned}
\label{eq:rank_Q}
\end{equation}
Now, let us assume that all items $2$-$5$ from the lemma statement are false.
For the first term of (\ref{eq:rank_Q}), we have
\begin{align*}
& - \beta^2/(2\alpha) \left\|\vPi_{\im(\UU^*)}\WW \AA\vPi_{\im(\VV^*)}\right\|_F^2\\
& \geq - \beta^2/(2\alpha) \left\|\vPi_{\im(\UU^*)}\WW \AA\right\|_F^2\\
& = - \beta^2/(2\alpha) \langle \vPi_{\im(\UU^*)}, \WW \AA\AA^\top \WW\rangle\\
& \geq - (\beta/20) \langle \WW, \AA\AA^\top \rangle\,,
\end{align*}
where we used item $2$ from the lemma statement, and
similarly for the second term of (\ref{eq:rank_Q}),
\begin{align*}
& - \beta^2/(2\alpha) \left\|\vPi_{\im(\UU^*)}\AA\YY\vPi_{\im(\VV^*)}\right\|_F^2\\
& \geq - (\beta/20) \langle \YY, \AA^\top \AA \rangle\,.
\end{align*}
Now we look at the second to last term of (\ref{eq:rank_Q}), i.e.
\begin{align*}
& - (\beta/2) \langle\WW \AA, \vPi_{\im(\UU^*)}\AA\vPi_{\im(\VV^*)} \rangle\\
& = - (\beta/2) \langle\WW \AA\vPi_{\im(\VV^*)} \AA^\top, \vPi_{\im(\UU^*)} \rangle \,.
\end{align*}
Now, we use the matrix Holder inequality 
\begin{align*}
& - (\beta/2) \langle\WW \AA\vPi_{\im(\VV^*)} \AA^\top, \vPi_{\im(\UU^*)} \rangle \\
& \geq - (\beta/2) \left\|\WW \AA\vPi_{\im(\VV^*)} \AA^\top\right\|_*\left\|\vPi_{\im(\UU^*)}\right\|_2 \\
& \geq - (\beta/2) \left\|\WW \AA\vPi_{\im(\VV^*)} \AA^\top\right\|_*\,,
\end{align*}
which can be proved by applying von Neumann's trace 
inequality and then the classical Holder inequality. Now, note that 
the matrix $\WW \AA\vPi_{\im(\VV^*)} \AA^\top$ is similar to
$\WW^{1/2} \AA\vPi_{\im(\VV^*)} \AA^\top \WW^{1/2}$, and so they have the same eigenvalues.
Furthermore, the latter is a symmetric PSD matrix, and so the former has real positive eigenvalues
as well. This means that its singular values are the same as its eigenvalues, and as a result
the nuclear norm is equal to the trace, i.e. 
\begin{align*}
& - (\beta/2) \left\|\WW \AA\vPi_{\im(\VV^*)} \AA^\top\right\|_*\\
& = - (\beta/2) \, \mathrm{Tr}\left(\WW \AA\vPi_{\im(\VV^*)} \AA^\top\right)\\
& = - (\beta/2) \langle \vPi_{\im(\VV^*)}, \AA^\top \WW \AA\rangle\\
& \geq - (\beta/2) \, \mathrm{Tr}\left[H_r\left(\AA^\top \WW \AA\right)\right]\\
& \geq - (1/5) \left(g(\AA) - f(\AA^*)\right)\,.
\end{align*}
where we also used Lemma~\ref{fact:matrix_ineq} and item $4$ from the lemma statement.
So we derived that
\begin{align*}
& - (\beta/2) \langle\WW \AA, \vPi_{\im(\UU^*)}\AA\vPi_{\im(\VV^*)} \rangle\\
& \geq - (1/5) \left(g(\AA) - f(\AA^*)\right)\,,
\end{align*}
and similarly for the last term of (\ref{eq:rank_Q}),
\begin{align*}
& - (\beta/2) \langle\AA\YY, \vPi_{\im(\UU^*)}\AA\vPi_{\im(\VV^*)} \rangle\\
& \geq - (1/5) \left(g(\AA) - f(\AA^*)\right)\,.
\end{align*}
Plugging the four inequalities that we derived back into (\ref{eq:rank_Q}), we get
\begin{align*}
 Q 
& \geq (\beta/2 - \beta/20) \langle\WW, \AA\AA^\top \rangle
{\ifdefined\arxiv
\else
\\ &
\fi}
+(\beta/2 - \beta/20) \langle\YY, \AA^\top\AA \rangle
{\ifdefined\arxiv
\else
\\ &
\fi}
- (2/5) \left(g(\AA) - f(\AA^*)\right)\\
& = (9/5) \Phi(\AA)
 - (2/5) \left(g(\AA) - f(\AA^*)\right)\\
& > (3/2) \Phi(\AA)
 - (2/5) \left(g(\AA) - f(\AA^*)\right)\,.
\end{align*}
Finally, combining this with the smoothness inequality (\ref{eq:rank_smooth}) and the
lower bound on $\left\|\nabla g(\AA)\right\|_2^2$ (\ref{eq:rank_gradient}),
we derive
\begin{align*}
& g(\AA') - g(\AA)\\
& \leq -(16\kappa r)^{-1} \Big(g(\AA) - f(\AA^*) + (1/2) \Phi(\AA) \Big)+ 2\beta \lambda_j^2\\
& = -(16\kappa r)^{-1} \Big(g(\AA) - f(\AA^*)\Big) - (32\kappa r)^{-1} \Phi(\AA)+ 2\beta \lambda_j^2\,.
\end{align*}
What remains is the bound the sum of the last two terms. 
We remind the reader that $\AA = \UU \vLambda \VV^\top$.
Now, letting $\zz$ equal to the vectorized diagonal of $\UU^\top \WW \UU$
and $\vlambda$ to the vectorized diagonal of $\vLambda$,
note that 
\[ \left\|\vlambda\right\|_{\zz}^2 = \langle \vLambda^2, \UU^\top \WW \UU\rangle = \langle \WW, \AA\AA^\top\rangle\,, \]
using which we derive
\begin{align*}
\lambda_j^2 = \min_{j\in S} \lambda_j^2 
& \leq \frac{\left\|\vlambda\right\|_{\zz}^2}{\left\|\zz\right\|_1} \\
& = \frac{\langle \WW, \AA\AA^\top\rangle}{\mathrm{Tr}[\UU^\top \WW \UU]}\\
& = \frac{\langle \WW, \AA\AA^\top\rangle}{\mathrm{Tr}[\UU^\top\UU] - \mathrm{Tr}[\UU^\top(\II-\WW)\UU]}\\
& \leq \frac{\langle \WW, \AA\AA^\top\rangle}{r' - \mathrm{Tr}[\II - \WW]}\\
& \leq \frac{\langle \WW, \AA\AA^\top\rangle}{r'/2} \\
& \leq \frac{\langle \WW, \AA\AA^\top\rangle}{128r\kappa} \,,
\end{align*}
where we used the fact that 
\begin{align*}
& \mathrm{Tr}[\UU^\top (\II-\WW) \UU] \\
& = \mathrm{Tr}\left[(\II-\WW)^{1/2} \UU \UU^\top (\II-\WW)^{1/2} \right] \\
& \leq \mathrm{Tr}[\II-\WW]\,, 
\end{align*}
because the columns of $\UU$ are orthonormal. 
We also used the property that $\mathrm{Tr}[\II-\WW] \leq r' / 2$
and the fact that $r' \geq 256 r \kappa$ by the lemma statement. 

Similarly, we derive that 
\begin{align*}
\lambda_j^2 \leq \frac{\langle \YY, \AA^\top \AA\rangle}{128r\kappa} \,,
\end{align*}
and, adding these two inequalities, we have
\begin{align*}
2\beta \lambda_j^2 \leq (32r\kappa)^{-1} \Phi(\AA)\,,
\end{align*}
finally concluding that 
\begin{align*}
& g(\AA') - g(\AA) \leq -(16\kappa r)^{-1} \Big(g(\AA) - f(\AA^*)\Big)\,.
\end{align*}
\end{proof}

\ifx 0 
\begin{algorithm}%
   \caption{Regularized symmetric low rank minimization}
   \label{alg:rank_ARHT_plus}
\begin{algorithmic}
\STATE $\beta$: smoothness bound, $T$: \#iterations
\STATE $\AA^0$: initial solution of rank $r$
\STATE $\eta = 1/(2\beta)$
\STATE $\WW^0 = \II$
\FOR{$t=0\dots T-1$}
	\STATE Let $\lambda \uu \uu^\top = H_1\left(\eta \nabla f(\AA^t) + 0.5 \WW^t \AA^t\right)$
	\STATE $\AA^{t+1} = \left(\II - 0.5 \WW^t\right) \AA^t - \eta H_1\left(\nabla f(\AA^t)\right)$

	\STATE Let $\MM = H_{s^*}\left((\WW^{t-1})^{1/2} (\AA^{t-1})^2 (\WW^{t-1})^{1/2}\right)$
	\IF{\[\mathrm{Tr}\left[\MM\right] > (1/8)\langle \WW^{t-1}, (\AA^{t-1})^2\rangle\] and 
	\[\langle \WW^{t-1}, (\AA^{t-1})^2\rangle \geq (2\beta)^{-1} (g(\AA^{t-1}) - f(\AA^*))\]}
		\STATE $\WW^t = (\WW^{t-1})^{1/2} \left(\II - (s^*)^{-1} \vPi_{\im(\MM)}\right) (\WW^{t-1})^{1/2}$
		\STATE $\AA^t = \AA^{t-1}$
		\STATE {\bf continue}
	\ENDIF
	\STATE $\oAA = (\II-\WW^{t-1}/2) \AA^{t-1} - \eta \cdot \nabla f(\AA^{t-1})$
	\STATE Let $g(\AA) = f(\AA) + (\beta/2)\left\langle \WW, \AA^2\right\rangle$
	\STATE Let $\vPi = \vPi_{\im(\AA^{t-1})}$
	\STATE $\AA^{\text{new}} = H_{s-s^*}\left(\vPi \oAA \vPi\right) + H_{s^*}\left(\oAA - \vPi\oAA\vPi\right)$
	\IF {$g(\AA^{\text{new}}) \leq 
	g(\vPi \oAA \vPi)$}
		\STATE $\AA^t = \AA^{\text{new}}$
	\ELSE
		\STATE $\AA^t = \vPi \oAA \vPi$
	\ENDIF
	\STATE $\WW^{t} = \WW^{t-1} - c \cdot \WW \AA^2 \WW / \langle \WW, \AA^2\rangle$
\ENDFOR
\end{algorithmic}
\end{algorithm}

\begin{algorithm}%
   \caption{ARHT+ for symmetric low rank minimization}
   \label{alg:rank_ARHT_plus}
\begin{algorithmic}
\STATE $\beta$: smoothness bound, $T$: \#iterations
\STATE $\AA^0$: initial solution of rank $r$
\STATE $\eta = 1/(2\beta)$, $c = s / T$
\STATE $\WW^0 = \II$
\FOR{$t=1\dots T$}
	\STATE Let $\MM = H_{s^*}\left((\WW^{t-1})^{1/2} (\AA^{t-1})^2 (\WW^{t-1})^{1/2}\right)$
	\IF{\[\mathrm{Tr}\left[\MM\right] > (1/8)\langle \WW^{t-1}, (\AA^{t-1})^2\rangle\] and 
	\[\langle \WW^{t-1}, (\AA^{t-1})^2\rangle \geq (2\beta)^{-1} (g(\AA^{t-1}) - f(\AA^*))\]}
		\STATE $\WW^t = (\WW^{t-1})^{1/2} \left(\II - (s^*)^{-1} \vPi_{\im(\MM)}\right) (\WW^{t-1})^{1/2}$
		\STATE $\AA^t = \AA^{t-1}$
		\STATE {\bf continue}
	\ENDIF
	\STATE $\oAA = (\II-\WW^{t-1}/2) \AA^{t-1} - \eta \cdot \nabla f(\AA^{t-1})$
	\STATE Let $g(\AA) = f(\AA) + (\beta/2)\left\langle \WW, \AA^2\right\rangle$
	\STATE Let $\vPi = \vPi_{\im(\AA^{t-1})}$
	\STATE $\AA^{\text{new}} = H_{s-s^*}\left(\vPi \oAA \vPi\right) + H_{s^*}\left(\oAA - \vPi\oAA\vPi\right)$
	\IF {$g(\AA^{\text{new}}) \leq 
	g(\vPi \oAA \vPi)$}
		\STATE $\AA^t = \AA^{\text{new}}$
	\ELSE
		\STATE $\AA^t = \vPi \oAA \vPi$
	\ENDIF
	\STATE $\WW^{t} = \WW^{t-1} - c \cdot \WW \AA^2 \WW / \langle \WW, \AA^2\rangle$
\ENDFOR
\end{algorithmic}
\end{algorithm}
\fi

%% file: appendix_lower_bound.tex
\section{Lower Bounds}
\label{sec:lower_bounds}

\begin{lemma}[IHT lower bound]
Let $f(\xx) := (1/2) \left\|\AA\xx-\bb\right\|_2^2$. For any $\kappa,s \geq 1$, $s' \leq 0.6 s \kappa^2$,
there exists a (diagonal) matrix $\AA\in\mathbb{R}^{n\times n}$ and a vector $\bb\in\mathbb{R}^n$
where $n = s(\kappa^2+\kappa+1)$,
$f$ is $1$-strongly convex and $\kappa$-smooth,
as well as an
$s$-sparse solution $\xx^*$ and 
an $s'$-sparse solution $\xx$, such that
\[ f(\xx) \geq f(\xx^*) + 0.1 s \kappa^2 \]
but
\[ \xx = H_{s'}\left(\xx - \beta^{-1} \nabla f(\xx)\right)\,, \]
i.e. $\xx$ is a fixpoint for IHT.
\end{lemma}
\begin{proof}
We use the same example as in~\cite{axiotis2021sparse}, Section 5.2: $\AA$ is diagonal with
\begin{align*}
\AA_{ii} = \begin{cases}
1 & \text{if $i\in I_1$}\\
\sqrt{\kappa} & \text{if $i\in I_2$}\\
1 & \text{if $i\in I_3$}\,,
\end{cases}
\end{align*}
where $I_1=[s],I_2=[s+1,s(\kappa+1)],I_3=[s(\kappa+1)+1,s(\kappa^2+\kappa+1)]$, and $\bb$ is defined as
\begin{align*}
b_i = \begin{cases}
\kappa \sqrt{1-4\delta} & \text{if $i\in I_1$}\\
\sqrt{\kappa}\sqrt{1-2\delta} & \text{if $i\in I_2$}\\
1 & \text{if $i\in I_3$}\,,
\end{cases}
\end{align*}
for some sufficiently small $\delta>0$ used for tie-breaking. We define
\begin{align*}
x_i^* = \begin{cases}
\kappa\sqrt{1-4\delta} & \text{if $i\in I_1$}\\
0 & \text{otherwise}
\end{cases}
\end{align*}
and, for some arbitrary $s'$-sized $S\subseteq I_3$
\begin{align*}
x_i = \begin{cases}
0 & \text{if $i\in I_1\cup I_2\cup I_3\backslash S$}\\
1 & \text{otherwise}\,.
\end{cases}
\end{align*}
Note that $f(\xx) - f(\xx^*) = 0.5 s \kappa^2 (1-4\delta) - 0.5 s' \geq 0.1 s\kappa^2$. Furthermore, the gradient is equal to
\begin{align*}
 \nabla f(\xx) 
& = \AA^\top\left(\AA\xx - \bb\right)\\
& = \begin{cases}
-\kappa\sqrt{1-4\delta} & \text{if $i\in I_1$}\\
-\kappa\sqrt{1-2\delta} & \text{if $i\in I_2$}\\
-1 & \text{if $i\in I_3\backslash S$}\\
0 & \text{if $i\in S$}\,,
\end{cases}
\end{align*}
and since we have $\beta = \kappa$,
\begin{align*}
\xx - \beta^{-1} \nabla f(\xx)
=\begin{cases}
\sqrt{1-4\delta} & \text{if $i\in I_1$}\\
\sqrt{1-2\delta} & \text{if $i\in I_2$}\\
1/\kappa & \text{if $i\in I_3\backslash S$}\\
1 & \text{if $i\in S$}\,,
\end{cases}
\end{align*}
implying that $H_{s'}\left(\xx - \beta^{-1} \nabla f(\xx)\right) = \xx$.
\end{proof}